\documentclass{article}
\usepackage{amsmath}
\usepackage{amsthm}
\usepackage{amssymb}
\usepackage{graphicx}
\usepackage{graphics}
\usepackage{listings}
\usepackage{multirow}
\usepackage{framed}
\usepackage{hyperref}

\DeclareMathOperator{\conv}{conv}

\usepackage{tikz}
\usetikzlibrary{arrows,shapes,graphs,fit,patterns}

\usepackage{color}

\usepackage[algo2e,ruled,vlined,]{algorithm2e}
\usepackage{algorithmic}

\definecolor{red}{rgb}{1,0,0}

\definecolor{blue}{rgb}{0,0,1}

\usepackage{mathrsfs}

\newcommand{\ms}{\medskip}
\newcommand{\bpf}{\begin{proof}}
\newcommand{\epf}{\end{proof}\ms}

\usepackage{lineno}
\newtheorem{theorem}{Theorem}
\newtheorem{corollary}[theorem]{Corollary}
\newtheorem{lemma}[theorem]{Lemma}
\newtheorem{proposition}[theorem]{Proposition}
\newtheorem{observation}[theorem]{Observation}

\theoremstyle{definition}

\newtheorem{conjecture}{Conjecture}
\newtheorem{claim}{Claim}

\begin{document}

\title{Intersections, circuits, and colorability of line segments}

\author{Boris Brimkov\thanks{Department of Computational and Applied Mathematics, Rice University, Houston, TX, 77005, USA (boris.brimkov@rice.edu)} \hskip 1.5em
Jesse Geneson\thanks{Department of Mathematics, Iowa State University, Ames, IA, 50011, USA (geneson@gmail.com)} \hskip 1.5em
Alathea Jensen\thanks{Department of Mathematics and Computer Science, Susquehanna University, Selinsgrove, PA, 17870, USA (jensena@susqu.edu)}\\
Jordan Miller\thanks{Department of Mathematics and Statistics, Washington State University, Pullman, WA, 99164, USA (jordan.a.miller@wsu.edu)} \hskip 1.5em
Pouria Salehi Nowbandegani\thanks{Department of Mathematics, Vanderbilt University, Nashville, TN, 37235, USA (pouria.salehi.nowbande\-gani@vanderbilt.edu)}}

\date{}

\maketitle

\begin{abstract}
We derive sharp upper and lower bounds on the number of intersection points and closed regions that can occur in sets of line segments with certain structure, in terms of the number of segments. We consider sets of segments whose underlying planar graphs are Halin graphs, cactus graphs, maximal planar graphs, and triangle-free planar graphs, as well as randomly produced segment sets. We also apply these results to a variant of the Erd\H{o}s-Faber-Lov\'asz (EFL) Conjecture stating that the intersection points of $m$ segments can be colored with $m$ colors so that no segment contains points with the same color. We investigate an optimization problem related to the EFL Conjecture for line segments, determine its complexity, and provide some computational approaches. 
\smallskip

\end{abstract}

\section{Introduction}
Sets of straight line segments with special structures and properties appear in various applications of geometric modeling, such as scientific visualization, computer-aided design, and medical image processing. In the present paper, we investigate geometric and graph theoretic properties of segment sets with special structure. 

Let $M$ be a finite set of line segments of nonzero length drawn in the plane. Collinear intersecting segments will be treated as a single segment. Let $P(M)$ be the set of all intersection points of segments in $M$ and $J(M)$ be the set of endpoints of segments in $M$; note that $P(M) \cap J(M)$ may be non-empty. Let $G_M=(V(G_M),E(G_M))$ be the graph whose vertex set is $P(M)\cup J(M)$ and where vertices $u$ and $v$ are adjacent whenever there is a segment $s\in M$ which contains $u$ and $v$, such that there is no $w \in (P(M)\cup J(M)) \cap s$ that is between $u$ and $v$. Note that $G_M$ is a planar graph; unless otherwise stated, we will assume $G_M$ is endowed with the plane embedding specified by the drawing of $M$.

Let $C(M)$ be the set of inclusion-minimal closed simple polygonal curves in $M$ (equivalently, the set of bounded faces of $G_M$); we will call the elements of $C(M)$ \emph{circuits}. By a {\em circuit segment set} of $M$ we will mean the set of segments in $M$ that contribute to a circuit of $M$ by infinitely many points. There is a one-to-one correspondence between the circuits of $M$, the circuit segment sets of $M$, and the bounded faces of $G_M$. We will call a path-connected component of $M$ \emph{trivial} if it consists of a single segment, and \emph{nontrivial} if it contains two or more segments. Given a segment $s\in M$, $M\backslash s$ denotes removing all non-intersection points of $s$ from $M$. Similarly, given a subset of segments $S=\{s_1,\ldots,s_k\}\subset M$, $M\backslash S$ denotes $M\backslash s_1\backslash \ldots \backslash s_k$. We will also define $m(M)=|M|$, $p(M)=|P(M)|$, $j(M)=|J(M)|$, $c(M)=|C(M)|$, $n(M)=|V(G_M)|$, $e(M)=|E(G_M)|$, and $k_1(M)$ and $k_2(M)$ respectively as the number of trivial and nontrivial components of $M$. 
When there is no scope for confusion, dependence on $M$ in all definitions will be omitted. We will also use analogous definitions when $M$ is a set of simple open curves, or a set of lines; in the latter case, $J(M)=\emptyset$. 
The following are basic relations between the quantities defined above. 
\begin{observation}
\label{obs1}
For any segment set $M$,
\begin{enumerate}
\item $p\leq \binom{m}{2}$
\item $c\leq \binom{m-1}{2}$
\item $e\geq m$
\item $n\leq p+j$
\item $m\geq k_1+2k_2$
\item $p\geq k_2$.
\end{enumerate}
All bounds are sharp, i.e., there are classes of segment sets for which the bounds hold with equality. 
\end{observation}

Since every planar graph has an embedding where its edges are mapped to straight line segments (cf. \cite{wagner}), for any planar graph $G$ there exists a segment set $M$ such that $G_M\simeq G$ (any edges incident to a degree 2 vertex $v$ which are drawn as collinear segments in a straight-line embedding of $G$ can be slightly shifted so that $v$ becomes an intersection point). Thus, there is an equivalence between sets of line segments and planar graphs. Given a family $\mathcal{F}$ of planar graphs, we will refer to the family $\{M \text{ is a set of segments}:G_M\in \mathcal{F}\}$ as \emph{segment}-$\mathcal{F}$. For instance, any set of segments with $c=0$ will be called a \emph{segment forest}, and any set of segments which is path-connected and has $c=0$ will be called a \emph{segment tree}. Similarly, we will refer to the family $\{M \text{ is a set of lines}:G_M\in \mathcal{F}\}$ as \emph{line}-$\mathcal{F}$. 
For special classes of segment sets, the bounds from Observation \ref{obs1} can be improved. The following are sharp bounds for some simple families of segment sets.

\begin{observation} 
\label{obs2}
\begin{enumerate}
\item[]
\item If $M$ is a segment tree, $p\leq  m-1$ and $c=0$.    
\item If $M$ is a segment forest, then $p \leq m-k_1-k_2$ and $c=0$.
\item If $M$ is a segment unicyclic graph, then $p\leq m$ and $c=1$.
\end{enumerate}
\vspace{7pt}
\noindent Moreover, all bounds are sharp.
\end{observation}

\noindent As an example of a nontrivial result of this kind, Poonen and Rubinstein \cite{poonen} computed $p$ and $c$ for a set of segments formed by the diagonals of a regular polygon. A conference version of the present paper \cite{brimkov_iwcia} computed $p$ and $c$ for segment cactus graphs. See also \cite{dujmovic,durocher,green,igamberdiev,oliveros,samee} for related questions on drawing planar graphs with few segments, and combinatorial properties of sets of lines and segments.

These kinds of bounds can be used to analyze the time and space complexity of algorithms for finding the  intersections and bounded regions occurring in a set of segments in terms of $m$, $p$, and $c$; these are fundamental tasks in computational geometry and have been widely studied (cf. \cite{balaban,bentleyOttmann,chazelle,PreparataShamos,tiernan}). For example, the algorithms of Bentley-Ottmann \cite{bentleyOttmann}, Chazelle \cite{chazelle}, and Balaban \cite{balaban}, which compute all intersections in a given set of segments, have respective time complexities of $O((m+p) \log m)$, $O(p + \frac{m \log^2 m}{\log \log m})$, and $O(p+m \log m)$, the last one being optimal for general segment sets. The worst case performance of these algorithms is achieved for sets of segments with $\Omega(m^2)$ intersections, and is respectively $\Omega(m^2 \log m)$ for Bentley-Ottmann's algorithm, and $\Omega(m^2)$ for Chazelle's and Balaban's algorithms. However, as we show in the sequel, segment sets which feature a Halin or cactus structure have $p=O(m)$; thus, for these types of segment sets, Bentley-Ottmann's and Balaban's algorithms run in $O(m \log m)$ time and are superior to Chazelle's algorithm, which runs in $O(\frac{m \log^2 m} {\log \log m})$ time. As another example, Chen and Chan \cite{chenChan} and Vahrenhold et al. \cite{bose,vahrenhold} presented in-place algorithms for finding all intersections in a set of segments (i.e., algorithms which use $O(1)$ cells of memory in addition to the input array). The time complexity of these algorithms is $O((m +p) \log m)$ and $O(m \log^2 m + p)$, respectively; for arbitrary segment sets, Vahrenhold's algorithm is superior to Chen-Chan's algorithm, as they respectively require $\Omega(m^2)$ and $\Omega(m^2 \log m)$ time in the worst case. However, for the aforementioned classes of segment sets, Chen-Chan's algorithm runs in  $O(m \log m)$ time and is superior to Vahrenhold's algorithm which runs in $O(m \log^2 m)$ time. 

In some cases, there may be direct relations between $p$, $c$, and $m$. For example, in segment maximal outerplanar graphs\footnote{A graph is \emph{maximal outerplanar} if it has a plane embedding in which all vertices belong to the outer face, and adding any edge to the graph causes it to no longer have this property.}, $p=c+2$; however, while maximal outerplanar graphs with $n$ vertices have exactly $2n-3$ edges, a segment maximal outerplanar graph may be realized with far fewer segments. For example, the segment maximal outerplanar graph in Figure \ref{fig_segment_MOP} has $p=n=m=9$. This is possible because a segment can participate in arbitrarily many intersection points and circuits, while a graph edge is incident to exactly two vertices and two faces. See \cite{brevilliers,brimkov1,brimkov2,brimkov3,castro,chan,francis,kara,rappaport} and the bibliographies therein for other applications of computing $p$ and $c$, as well as for techniques and results on other problems defined on segment sets and on graphs constructed through segment sets.

\begin{figure}[h]
\begin{center}
\includegraphics[scale=0.25]{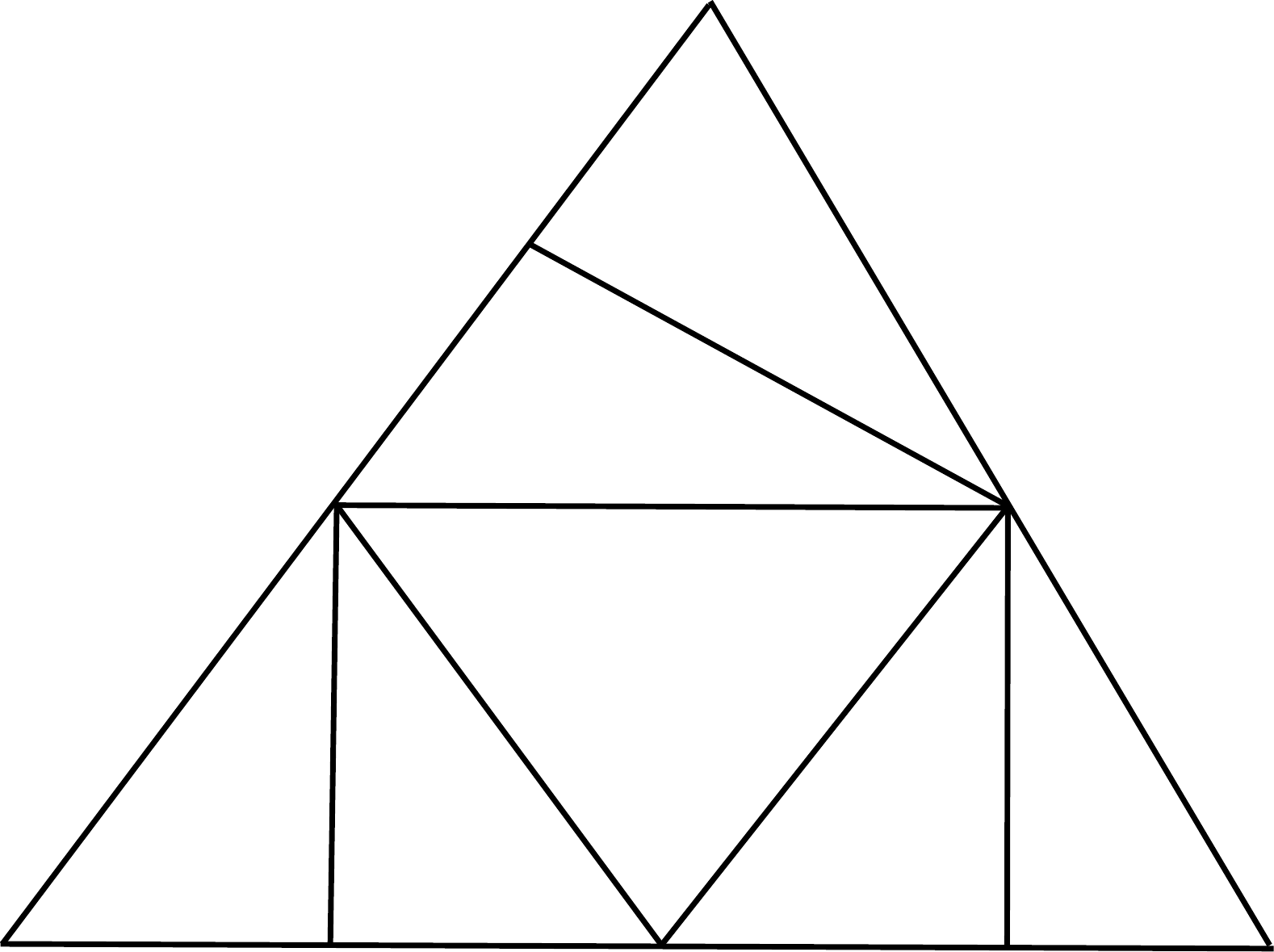}
\end{center}
\caption{A segment maximal outerplanar graph.}
\label{fig_segment_MOP}
\end{figure}

This paper is organized as follows. In the next section, we give some preliminary definitions and results. In Section \ref{section_bounds}, we give bounds on the number of intersections and circuits for various families of segment sets. In Section \ref{section_efl}, we formulate and explore a variant of the Erd\H{o}s-Faber-Lov\'asz Conjecture and a related optimization problem defined on sets of segments and lines. We end with some final remarks and directions for future work in Section \ref{section_conclusion}.

%A planar straight line graph is a graph in which the vertices are embedded as points in the Euclidean plane, and the edges %are embedded as non-crossing line segments.

\section{Preliminaries}

A \emph{cut vertex} of a graph $G$ is a vertex whose deletion increases the number of connected components of $G$. A \emph{biconnected component} or \emph{block} of $G$ is a maximal subgraph of $G$ which has no cut vertices. An isomorphism between graphs $G_1$ and $G_2$ will be denoted by $G_1\simeq G_2$. Given a vertex $v$ of $G$, $G-v$ will denote $G$ with $v$ removed, along with all edges incident to $v$. A vertex of $G$ is a \emph{leaf} if it has a single neighbor in $G$. $K_n$ denotes the complete graph on $n$ vertices.

%The \emph{block tree} of $G$ is the bipartite graph with parts $B$ and $C$, where $C$ is the set of cut vertices of $G$ and $B$ is the set of blocks of $G$; $c\in C$ is adjacent to $b \in B$ if and only if block $b$ contains cut vertex $c$. 

Let $M$ be a set of segments and $s$ be a segment of $M$ with endpoints $a$ and $b$. Let $a'$ be the first intersection point in $s$ encountered when moving along $s$ in a straight line from $a$ to $b$ in $M$, and $b'$ be the last intersection point encountered. \emph{Trimming} $s$ is the operation of replacing $s$ by a segment $s'$ with endpoints $a'$ and $b'$; if $s$ has fewer than two intersection points, then trimming $s$ means deleting $s$. \emph{Trimming} $M$ means repeatedly trimming the segments in $M$ until further trimming yields no difference. Note that it may be possible to trim a segment, then trim another segment, and then trim the first segment again. See Figure \ref{trim} for an illustration of trimming.

\begin{figure}[h]
\begin{center}
\includegraphics[scale=0.4]{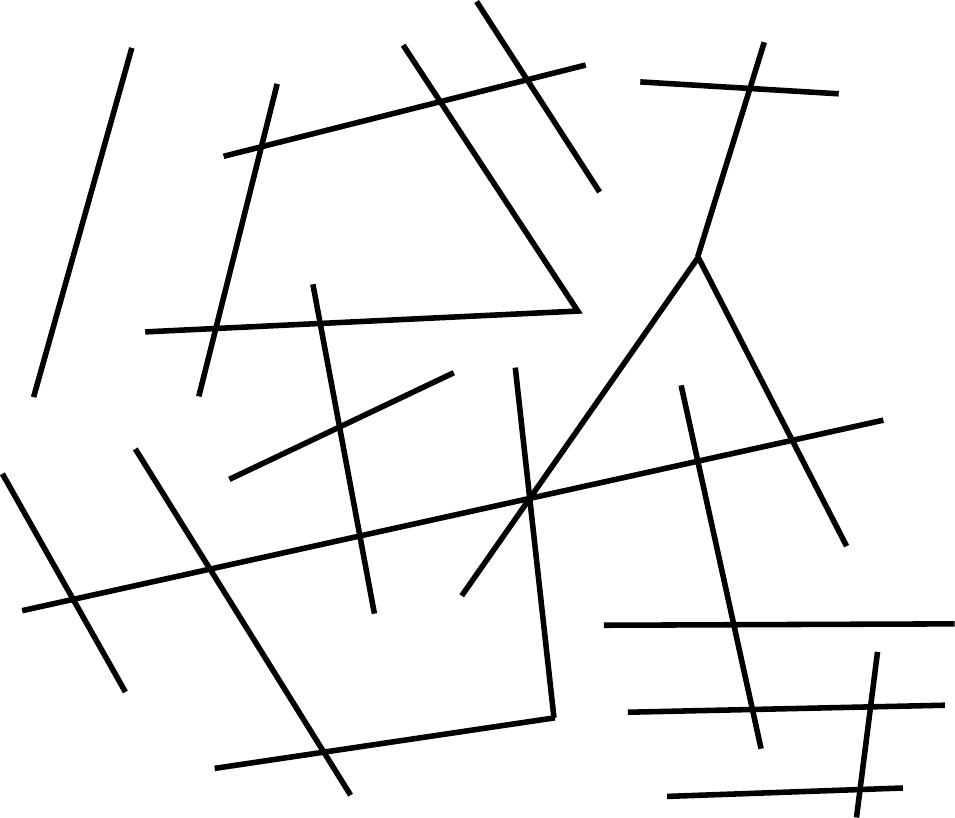} \qquad\quad
\includegraphics[scale=0.4]{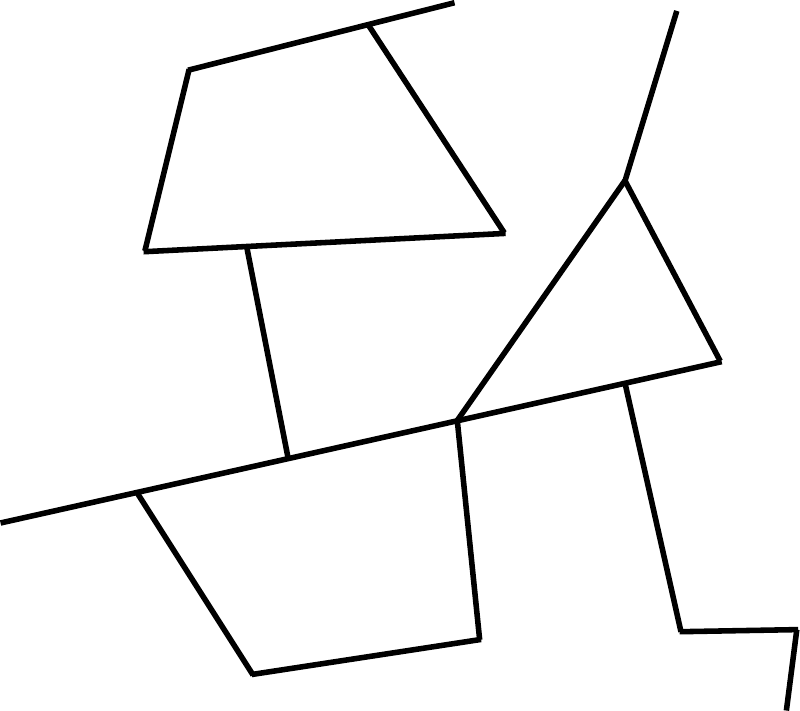} \qquad\quad
\includegraphics[scale=0.4]{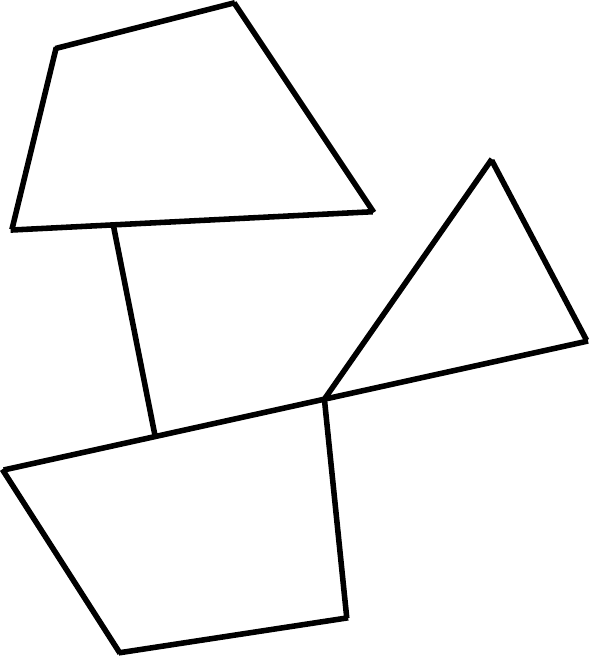}
\end{center}
\caption{\emph{Left:} Set of segments $M$. \emph{Middle:} Trimming every segment of $M$ once. \emph{Right:} Trimming $M$.}
\label{trim}
\end{figure}

\noindent We end this section with some preliminary observations about segment sets.

\begin{proposition}
For any nontrivial connected segment set $M$, there are at least two segments $s_a$ and $s_b$ in $M$ such that $M\backslash s_a$ and $M\backslash s_b$ are connected.
\label{prop1}
\end{proposition}
\proof
Let $H$ be a graph which has a vertex for each segment in $M$, and where two vertices are adjacent whenever the corresponding segments intersect in $M$. 

Let $s_x$ and $s_y$ be any two vertices of $H$, and $x$ and $y$ be non-intersection points respectively belonging to the segments $s_x$ and $s_y$ in $M$. Since $M$ is connected, there is a path $x, p_1,\ldots, p_k, y$ between $x$ and $y$, where $p_1,\ldots,p_k$ are parts of segments (or entire segments) of $M$. In particular, let $p_t\subseteq s_{i_t}$ for $1\leq t\leq k$ (where $s_{i_1}=s_x$  and $s_{i_k}=s_y$). By construction of $H$, for $1\leq t\leq k-1$, $s_{i_t}$ is adjacent to $s_{i_{t+1}}$ in $H$. Thus, the path $x, p_1,\ldots, p_k, y$ in $M$ corresponds to a path $s_x, s_{i_1},\ldots, s_{i_k}, s_y$ in $H$, so $H$ is connected.

Since any connected graph with at least two vertices has at least two non-cut vertices, $H$ has two non-cut vertices $s_a$ and $s_b$. We claim that $M\backslash s_a$ and $M\backslash s_b$ are connected. To see why, let $x$ and $y$ be any two points in $M\backslash s_a$. If $x$ and $y$ belong to the same segment, clearly there is a path between them. Otherwise, let $s_x$ and $s_y$ respectively be segments containing $x$ and $y$. Since $s_a$ is a non-cut vertex of $H$, $H-s_a$ is connected. Let $s_x,s_{i_1},\ldots,s_{i_k},s_y$ be a simple path between $s_x$ and $s_y$ in $H-s_a$. By construction of $H$, segments $s_x$ and $s_{i_1}$ intersect in $M$; thus, there is a path between $x$ and every point in $s_{i_1}$. Similarly, segments $s_{i_t}$ and $s_{i_{t+1}}$ intersect in $M$ for $1\leq t\leq k-1$, and $s_{i_k}$ intersects $s_y$, so there is a path between $x$ and $y$ in $M\backslash s_a$. Thus, $M\backslash s_a$ is connected; similarly, $M\backslash s_b$ is connected.
\qed

\begin{corollary}
Any nontrivial segment tree $M$ contains at least two segments $s_a$ and $s_b$ such that $M\backslash s_a$ and $M\backslash s_b$ are connected, and such that $s_a$ and $s_b$ each contain a single intersection point. 
\label{cor1}
\end{corollary}
\proof
By Proposition \ref{prop1}, there are two segments $s_a$ and $s_b$ such that $M\backslash s_a$ and $M\backslash s_b$ are connected; we claim that each of these segments contains a single intersection point. Indeed, since $M$ is a segment tree and is therefore connected, $s_a$ and $s_b$ must each contain at least one intersection point. Suppose for contradiction that $s_a$ contains two (or more) intersection points $x$ and $y$. Since $M$ is a segment tree, there is only one path, namely along $s_a$, between the segments which intersect $s_a$ at $x$ and the segments which intersect $s_a$ at $y$. Then, there will be no path between these segments in $M\backslash s_a$, a contradiction.
\qed
\vspace{9pt}

\section{Bounds on $p$ and $c$ }
\label{section_bounds}

In this section, we derive tight bounds on the number of intersection points and circuits in certain families of segment sets as a function of the number of segments.

\subsection{Segment Halin graphs}

A \emph{Halin graph} is a graph that can be obtained by starting from a tree with no vertices of degree two which is embedded in the plane, and connecting the leaves of this tree in a cycle according to their clockwise ordering specified by the embedding. Halin graphs have a unique embedding up to the choice of which face is the outer face. A \emph{segment Halin graph} is a set of segments $M$ satisfying the following two properties: 1) $G_M$ is a Halin graph; 2) in the embedding of $G_M$ induced by $M$, the edges of the outer face constitute the cycle used in the construction of $G_M$. See \cite{Halin1,Halin4,Halin3,Halin2} for some applications and algorithmic aspects of Halin graphs.

\begin{theorem}
Let $M$ be a segment Halin graph. Then
\[p(M)\geq\left\lceil\frac{m(M)+2}{2}\right\rceil,\]
\[c(M)\geq\left\lceil\frac{m(M)+3}{3}\right\rceil,\]
and these bounds are tight. 
\end{theorem}
\proof
Halin graphs do not have any leaves, so for any segment Halin graph $M$, $J\subset P$, and hence $p(M)=n(M)$. Furthermore, for any Halin graph $G$, $|E(G)|=|V(G)|-1+\ell(G)$, where $\ell(G)$ is the number of leaves of the tree used in the construction of $G$. Thus for any segment Halin graph $M$, $e(M)\leq 2(p(M)-1)$, so
\[p(M)\geq\frac{e(M)+2}{2}\geq\frac{m(M)+2}{2}.\]
Since $m(M)$ is an integer, this bound can be tightened to
$p(M)\geq\left\lceil\frac{m(M)+2}{2}\right\rceil$. This bound holds with equality when $m$ is even and $m\geq 6$ for the set of segments formed by the edges of a straight-line noncollinear embedding of a wheel graph on $\frac{m}{2}+1$ vertices; see Figure \ref{halinfigure1}, left. When $m$ is odd and $m\geq 7$, equality can be achieved by a similar construction, shown in Figure \ref{halinfigure1}, right.  Since there are no segment Halin graphs with fewer than six segments, the bound on $p(M)$ is tight for all $m$.

%the following construction, shown in Figure \ref{halinfigure1}, left: first, draw a regular $\frac{m}{2}$-gon; then, draw $\frac{m}{2}$ segments, each of which has one endpoint in a distinct intersection point of the $\frac{m}{2}$-gon, and the other endpoint at a point $q$ in the interior of the $\frac{m}{2}$-gon chosen in such a way that none of the segments are collinear (clearly such a point $q$ exists, e.g., if it is in the interior of a triangle formed by three consecutive intersection points of the $\frac{m}{2}$-gon). When $m$ is odd and $m\geq 7$, equality can be achieved by a similar construction, shown in Figure \ref{halinfigure1}, right.  Since there are no segment Halin graphs with fewer than six segments, the bound on $p(M)$ is tight for all $m$.

\begin{figure}[h]
\begin{center}
\includegraphics[scale=0.2]{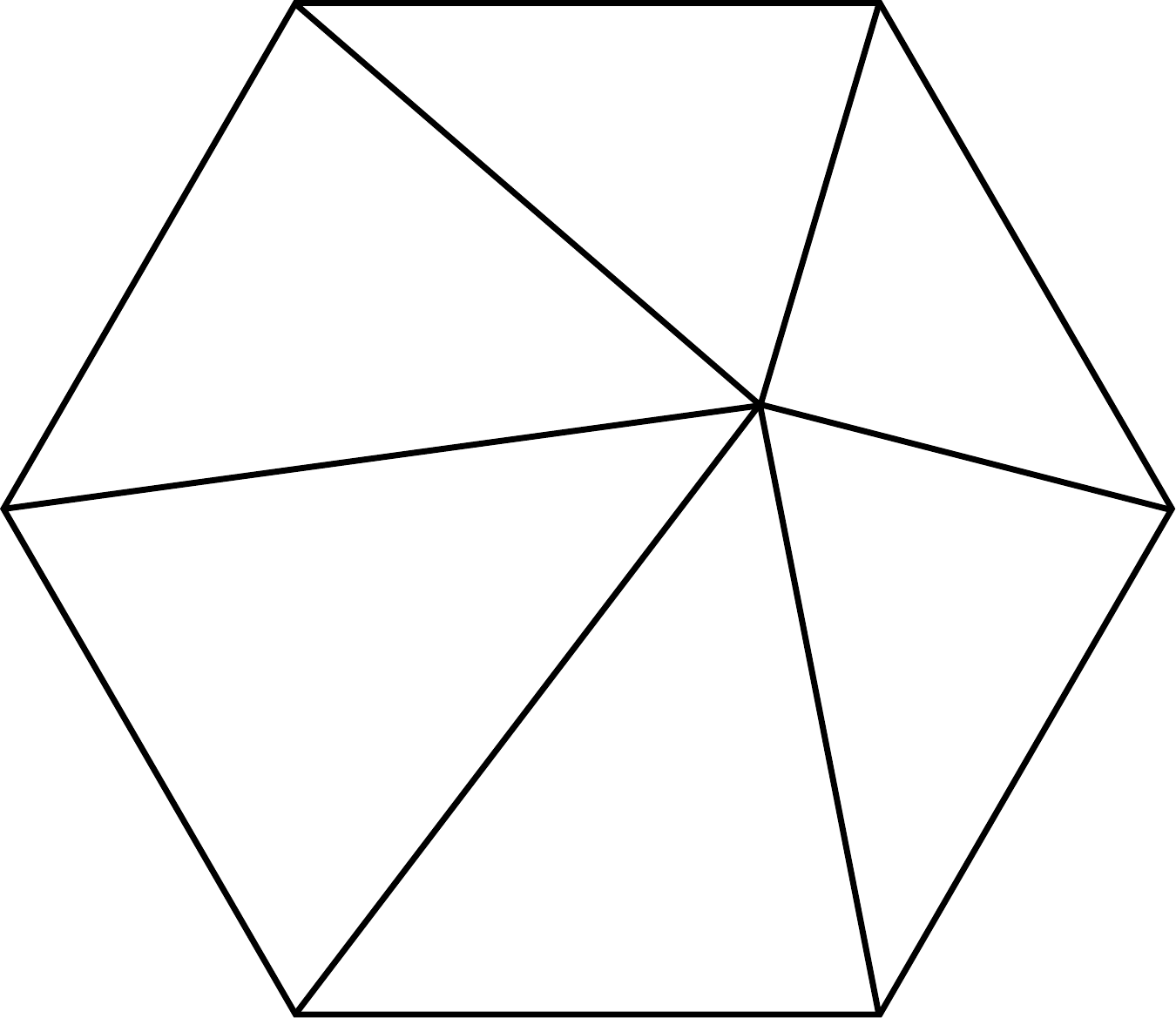} \qquad\qquad\qquad\quad
\includegraphics[scale=0.2]{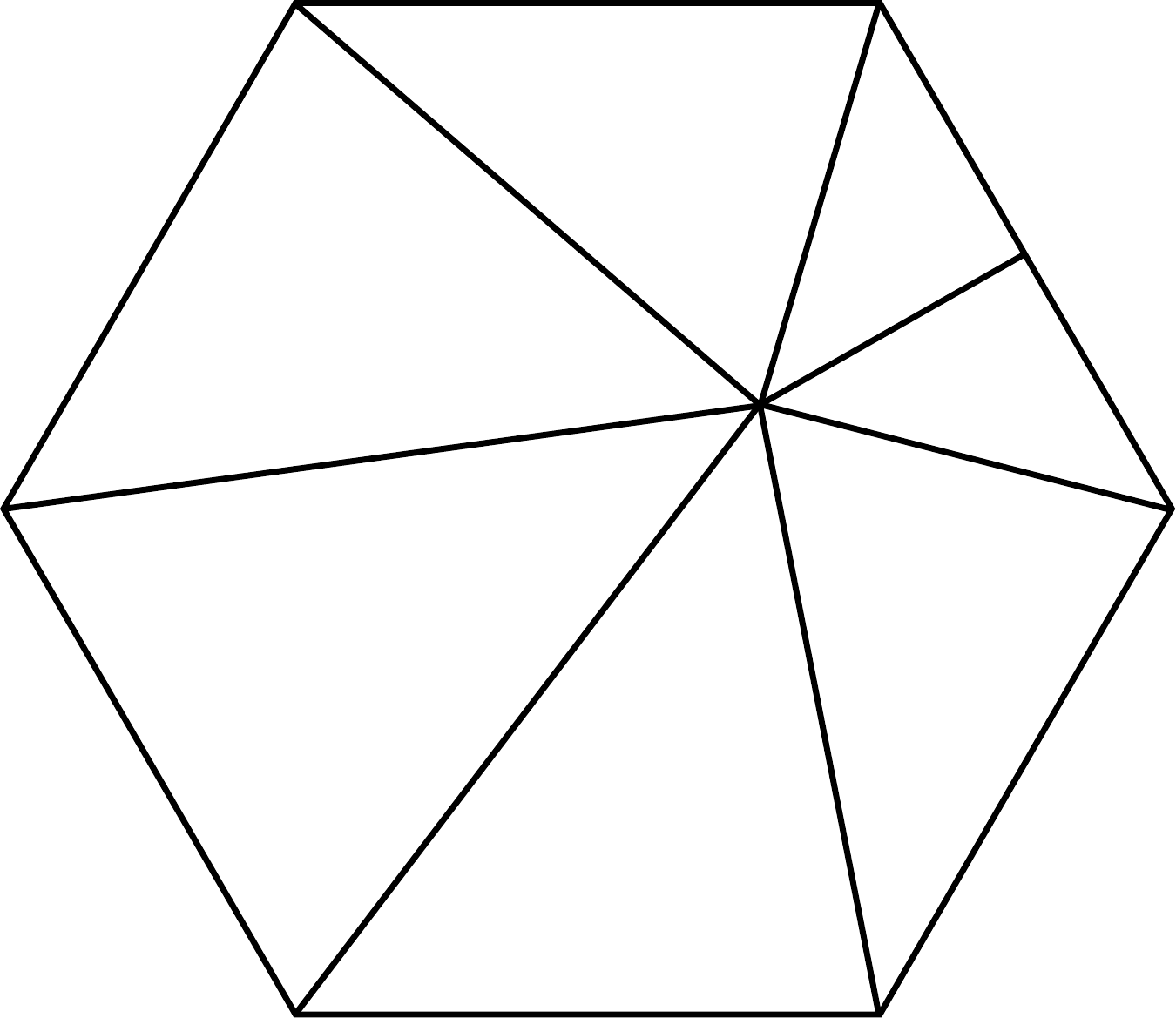} 
\end{center}
\caption{Constructions demonstrating the tightness of the lower bound on $p(M)$ for segment Halin graphs. On the left, $m(M)=12$; on the right, $m(M)=13$.}
\label{halinfigure1}
\end{figure}

Let $M$ be a segment Halin graph, $T$ be the tree used in the construction of $G_M$, and $\ell(T)$ be the number of leaves of $T$. Then, $|E(G_M)|=|E(T)|+|E(G_M)\backslash E(T)|=|V(T)|-1+\ell(T)$. 
Moreover, since by definition $T$ has no degree 2 vertices, its $\ell(T)$ leaves have degree 1, and its $|V(T)|-\ell(T)$ non-leaf vertices have degree at least 3.  Then,
\[2(|V(T)|-1)=2|E(T)|=\sum_{v\in V(T)}deg(v)\geq\ell(T)+3(|V(T)|-\ell(T))=
3|V(T)|-2\ell(T).\]
Solving for $|V(T)|$, we obtain $|V(T)|\leq 2\ell(T)-2$.
Moreover, the number of bounded faces of $G_M$ equals $\ell(T)$; thus, combining the inequalities above, we have
\[m(M)\leq |E(G_M)|\leq 3\ell(T)-3=3c(M)-3.\]
Solving for $c(M)$, we obtain $c(M)\geq\frac{m(M)+3}{3}$, and since $m(M)$ is an integer, this bound can be tightened to 
$c(M)\geq\left\lceil\frac{m(M)+3}{3}\right\rceil$.

The bound on $c(M)$ holds with equality when $m \equiv 0\ (\text{mod}\ 3)$ for the following construction, shown in Figure \ref{halinfigure2}, left: let $t=\frac{m}{3}+1$ and draw a regular $t$-gon $S$ with intersection points $p_1,\ldots,p_{t}$ in clockwise order, and a larger, dilated concentric copy of $S$ with intersection points $p_1',\ldots,p_{t}'$ in clockwise order; delete the segments $\overline{p_1p_2}$, $\overline{p_2p_3}$, $\overline{p_3p_4}$, and add the segments $\overline{p_1p_2'}$, $\overline{p_4p_3'}$, $\overline{p_1p_1'}$, and $\overline{p_kp_k'}$ for $4\leq k\leq t$. Similar constructions can be used in the cases when $m \equiv 2\ (\text{mod}\ 3)$ and $m \equiv 1\ (\text{mod}\ 3)$; see Figure \ref{halinfigure2}, center, and Figure \ref{halinfigure2}, right, respectively. Thus, the  bound on $c(M)$ is tight for all $m$.
\qed

\begin{figure}[h]
\begin{center}
\includegraphics[scale=0.2]{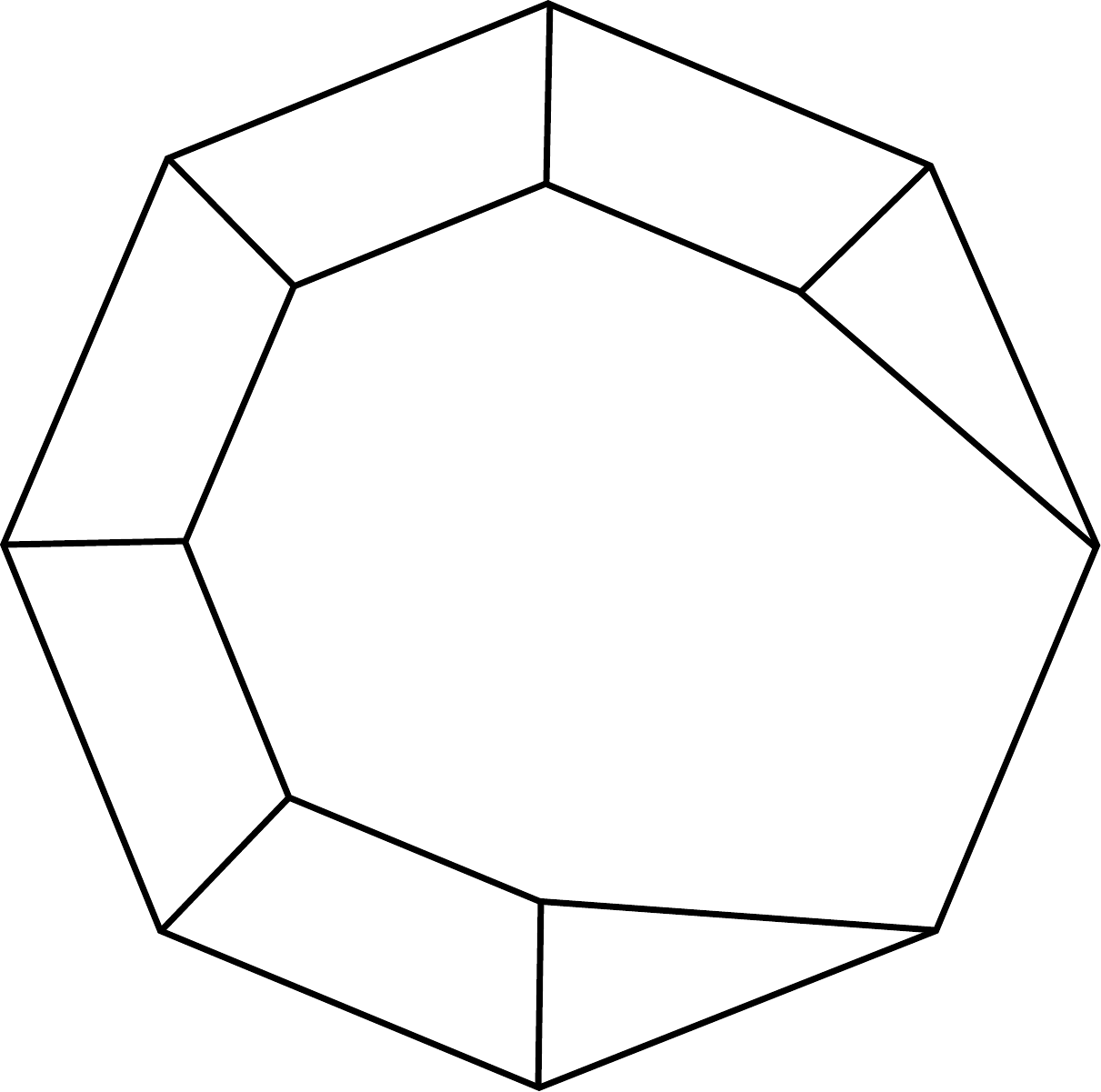} \qquad\qquad
\includegraphics[scale=0.2]{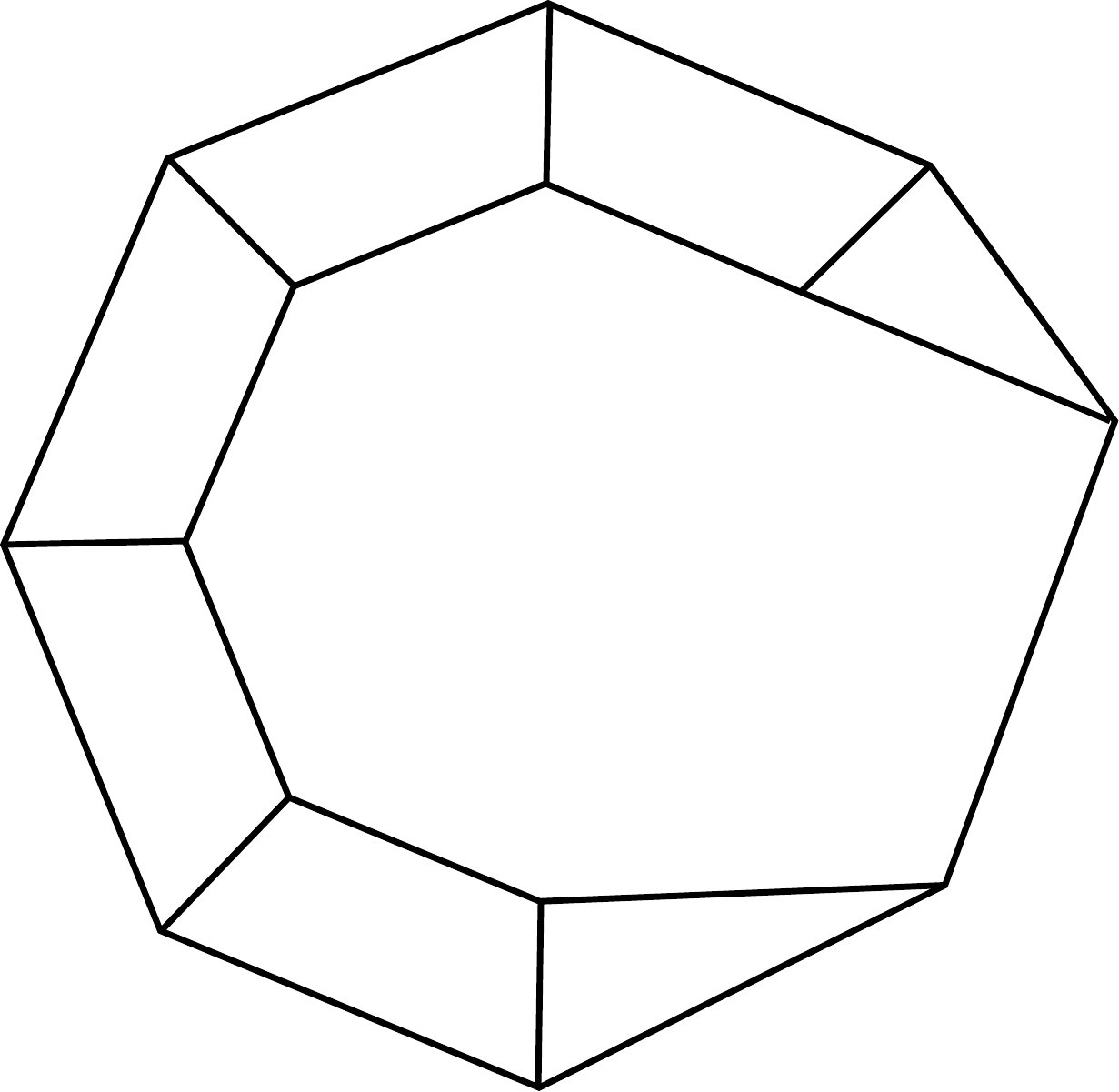} \qquad\qquad
\includegraphics[scale=0.2]{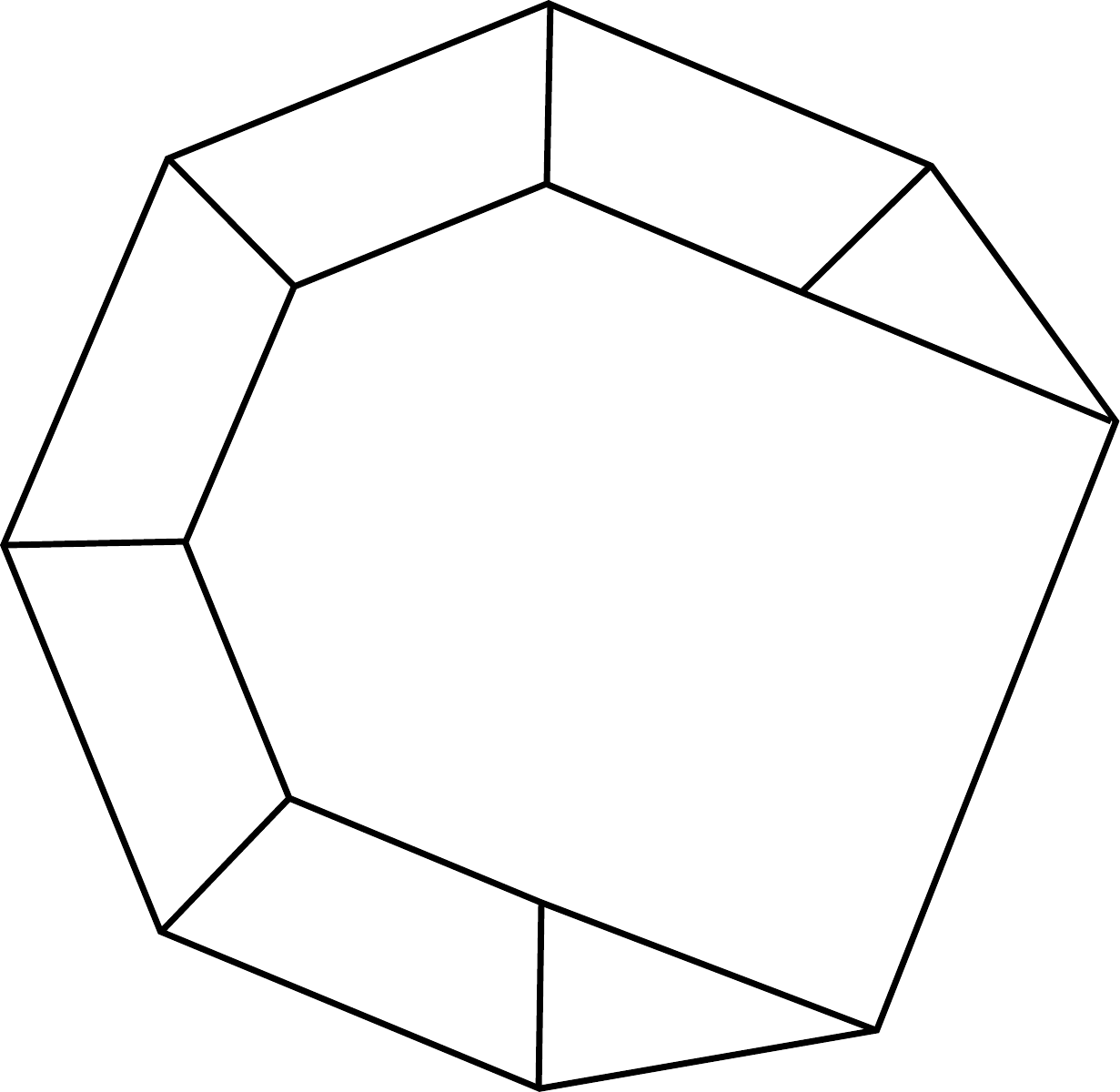} 
\end{center}
\caption{Constructions demonstrating the tightness of the lower bound on $c(M)$ for segment Halin graphs. On the left, $m(M)=21$; in the center, $m(M)=20$; on the right, $m(M)=19$.}
\label{halinfigure2}
\end{figure}

\begin{theorem}
Let $M$ be a segment Halin graph. Then
\[p(M)\leq 3m(M)-11,\]
\[c(M)\leq 2m(M)-6,\] 
and these bounds are tight.
\end{theorem}
\proof
Let $T$ be the tree used in the construction of $G_M$, let $\mathcal{C}$ be the cycle used in the construction of $G_M$ equipped with the embedding induced by $M$, and let $\ell(T)=|\mathcal{C}|$ denote the number of leaves of $T$ (equivalently, the order of $\mathcal{C}$). We will call vertices of $\mathcal{C}$ \emph{convex}, \emph{straight}, and \emph{concave} if their interior angle with respect to the embedding of $\mathcal{C}$ is respectively less than $\pi$, equal to $\pi$, and greater than $\pi$. Given a polygon with $k$ vertices, the sum of the interior angles of the vertices is $(k-2)\pi$. Thus, the polygon must contain at least 3 convex vertices, since if at most 2 of its $k$ vertices are convex, the sum of the interior degrees of the vertices would be greater than $(k-2)\pi$, a contradiction. 

We will first show that there are at least 3 segments in $M$ which belong exclusively to $\mathcal{C}$ and not to $T$. 
If a segment of $M$ belongs to both $\mathcal{C}$ and $T$, it must pass between $\mathcal{C}$ and $T$ at a concave vertex of $\mathcal{C}$. Moreover, since only leaves of $T$ touch $\mathcal{C}$, at most one of the segments that meet at a concave vertex can pass from $\mathcal{C}$ to $T$ at that vertex. Since $\mathcal{C}$ (equipped with its embedding induced by $M$) is a polygon, $\mathcal{C}$ must contains at least 3 convex vertices. Suppose $\mathcal{C}$ contains $r$ straight vertices. Then, the number of non-straight vertices of $\mathcal{C}$ is $|\mathcal{C}|-r$, and hence the number of segments that make up $\mathcal{C}$ is $|\mathcal{C}|-r$. Since at least 3 of the vertices of $\mathcal{C}$ are convex, at most $|\mathcal{C}|-r-3$ of the vertices of $\mathcal{C}$ are concave, so there are at most $|\mathcal{C}|-r-3$ places where a segment can pass from $T$ to $\mathcal{C}$. Then, at most $|\mathcal{C}|-r-3$ of the segments of $\mathcal{C}$ pass into $T$. Every segment that does not pass into $T$ is a segment that is contained entirely in $\mathcal{C}$. Thus, at least $3$ segments belong to $\mathcal{C}$ but not $T$, so 
\begin{equation}
m(T)\leq m(M)-3.
\label{eq_halin}
\end{equation}
Using the inequality in \eqref{eq_halin} and the fact that Halin graphs have no leaves, we have
\[p(M)=n(M)=n(T)\leq p(T)+j(T)\leq m(T)-1+2m(T)=3m(T)-1\leq 3(m(M)-3)-1.\] 
Suppose that $p(M)=3m(M)-10$. This can happen only if all of the following hold: 
\begin{enumerate}
\item[1)] $n(M)=p(T)+j(T)$,
\item[2)] $p(T)+j(T)=3m(T)-1$, 
\item[3)] $m(T)=m(M)-3$.  
\end{enumerate}

\noindent Equality 3) implies that exactly 3 segments of $M$ are exclusively in $\mathcal{C}$ and not in $T$. Hence, $|\mathcal{C}|-r-3$ of the segments of $\mathcal{C}$ must pass into $T$. Since segments can only pass between $\mathcal{C}$ and $T$ at a concave vertex of $\mathcal{C}$, and since at most one of the segments that meet at a concave vertex can pass from $\mathcal{C}$ to $T$ at that vertex, it follows that $|\mathcal{C}|-r-3$ of the non-straight vertices of $\mathcal{C}$ must be concave. Thus, there are exactly 3 convex vertices in $\mathcal{C}$. 

Let the convex vertices of $\mathcal{C}$ be $a_1,a_2,a_3$, and let $s_i$ be the segment of $T$ which has an endpoint at $a_i$, $1\leq i\leq 3$. Let $A_{12}$ be the path in $\mathcal{C}$ between $a_1$ and $a_2$ that does not pass through $a_3$; define $A_{13}$ and $A_{23}$ analogously. Equalities 1) and 2) imply that no intersection point of $T$ is also an endpoint of a segment of $T$, that no three segments of $T$ intersect in the same point, and that both endpoints of each segment of $T$ are leaves of $T$ and touch $\mathcal{C}$.
Since all vertices on $A_{12}$ and $A_{13}$ are concave or straight, $s_1$ cannot have its other endpoint on $A_{12}$ or $A_{13}$; thus, it must be in $A_{23}$. Similarly, $s_2$ must have its other endpoint in $A_{13}$. Thus, the segments $s_1$ and $s_2$ intersect in a point $x$ in the interior of $\mathcal{C}$. Moreover, $s_3$ must have its other endpoint in $A_{12}$, and must therefore intersect $s_1$ and $s_2$; see Figure \ref{halinfigure3}, left, for an illustration. However, if $s_3$ does not pass through $x$, then the segments $s_1,s_2,s_3$ form a triangle in the interior of $\mathcal{C}$; this triangle must be part of $T$, contradicting the fact that $T$ is a tree. On the other hand, if $s_3$ passes through $x$, this contradicts the fact that no three segments of $T$ intersect in the same point. Thus, 1), 2), and 3) cannot all hold at the same time, so $p(M)<3m(M)-10$. Since $p(M)$ is an integer, it follows that $p(M)\leq 3m(M)-11$. 

The number of bounded faces of $G_M$ equals $\ell(T)$; thus, again using the inequality \eqref{eq_halin}, we have $c(M)=\ell(T)\leq 2m(T)\leq 2(m(M)-3)$. The upper bounds on $p$ and $c$ are tight for all $m\geq 6$ for the family of segment Halin graphs shown in Figure \ref{halinfigure3}, right. 
\qed

\begin{figure}[h]
\begin{center}
\includegraphics[scale=0.2]{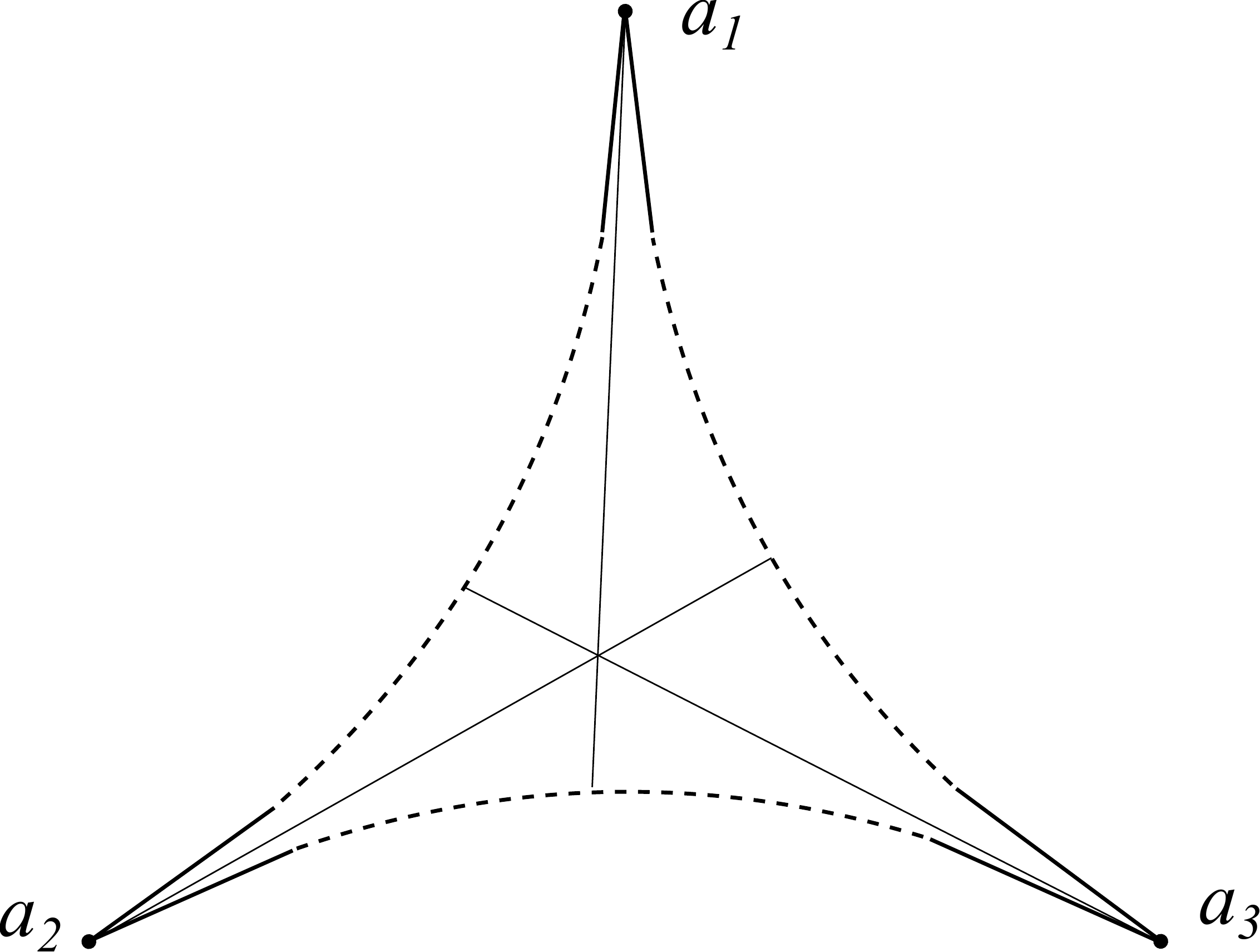}\hspace{100pt}
\includegraphics[scale=0.3]{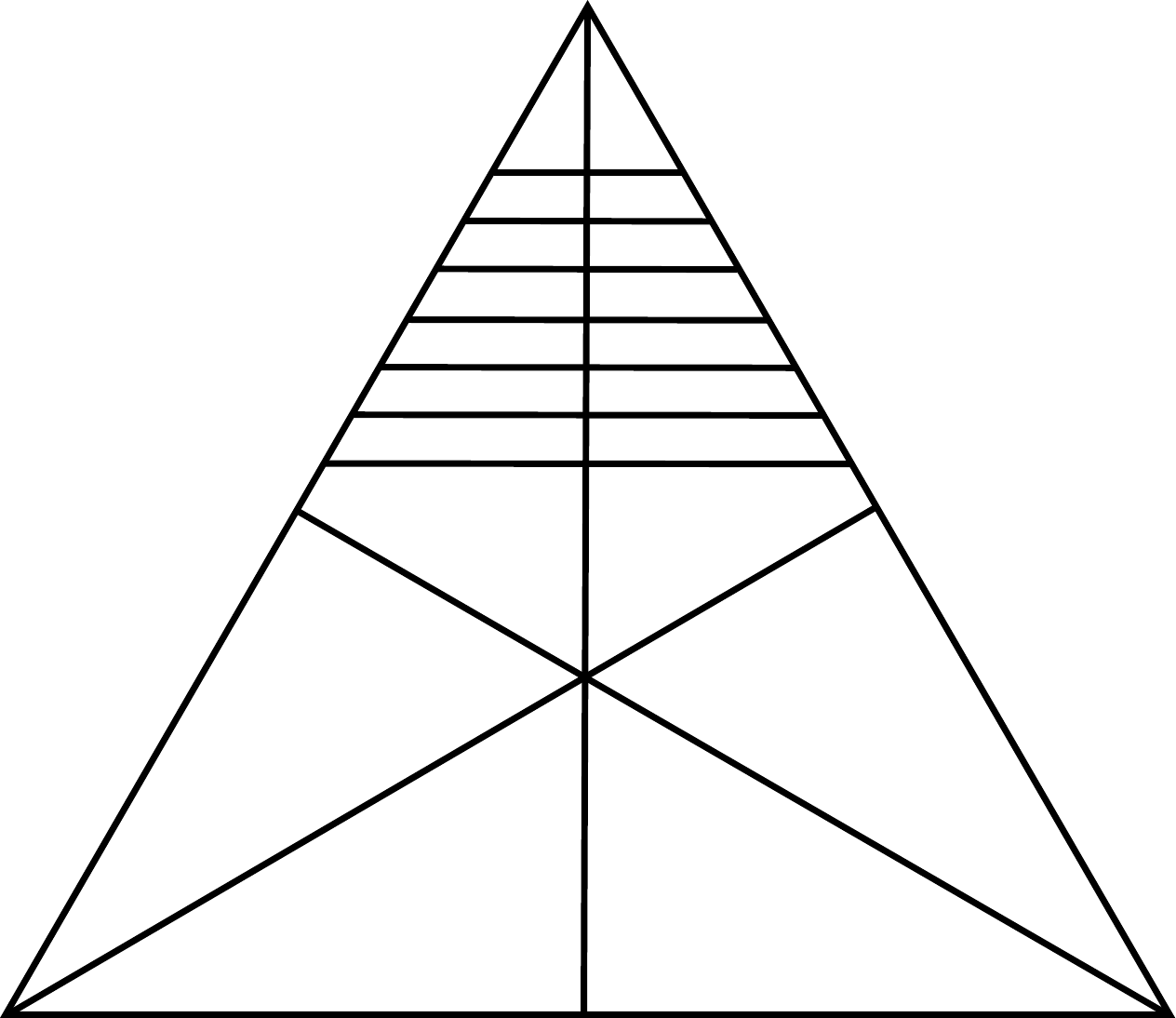}
\end{center}
\caption{\emph{Left}: The outer face of a Halin graph with exactly 3 convex points. \emph{Right}: Construction demonstrating the tightness of the upper bounds on $p(M)$ and $c(M)$ for segment Halin graphs.}
\label{halinfigure3}
\end{figure}

\subsection{Segment cactus graphs}

A graph $G$ is called a \emph{cactus} if any two cycles of $G$ have at most one vertex in common. Every edge of a cactus graph belongs to at most one cycle, and the biconnected components of a cactus graph are either cycles or single edges. By definition, two circuits of a segment cactus can have at most one vertex in common, i.e., they cannot share a portion of a segment different from a point. Properties of cactus graphs have been studied with some applications in mind; for example, cactus graphs arise in the design of telecommunication systems, material handling networks, and local area networks (cf. \cite{moshe,brimkov_cactus,harary,husimi,kariv,koontz} and the bibliographies therein). 

\begin{proposition}
A segment cactus $M$ with $c \geq 1$ circuits contains at least two segments $s_1$ and $s_2$, such that for $i\in \{1,2\}$,

\begin{itemize}
\item[$A)$] $s_i$ belongs to a single circuit segment set $S_i$,
\item[$B)$] the connected components of $M\backslash s_i$ which do not contain segments of $S_i$ are segment trees.
\end{itemize}
\label{main_prop}
\end{proposition}

\proof
If $c=1$, every segment in the single circuit segment set of $M$ satisfies properties $A)$ and $B)$; thus, assume henceforth that $c \geq 2$.

Let $Q=\{s_1,\ldots,s_q\}$ be a maximal set of segments of $M$ such that for $1\leq i\leq q$, $s_i$ does not belong to any circuit segment set of $M$, and $s_i$ is a segment whose deletion does not disconnect $M\backslash\{s_1,\ldots,s_{i-1}\}$. Let $M'=M\backslash Q$. By construction, $M$ and $M'$ have the same circuit segment sets; moreover, the connected components of $M\backslash M'$ (i.e. of $Q$) are segment trees. Hence, for any segment $s\in M'$, the connected components of $M\backslash s$ which do not contain segments of $M'$ are segment trees. Let $M''$ be the set of segments obtained by trimming $M'$ (in fact, $M''$ is identical to the set of segments obtained by trimming $M$). Note that $M$, $M'$, and $M''$ have the same circuits.

$G_{M''}$ has no leaves, since a leaf of $G_{M''}$ would have to be an endpoint of a segment in $M''$, and all endpoints of segments in $M''$ are also intersection points. Thus, all outer blocks of $G_{M''}$ (i.e., biconnected components with a single cut vertex) are cycles. Since $c\geq 2$ and since $M$ and $M''$ have the same circuits, it follows that $G_{M''}$ has at least two cycles; thus, $G_{M''}$ has at least two outer blocks which are cycles, say $C_1$ and $C_2$. Let $S_1$ and $S_2$ be the circuit segment sets in $M$ corresponding to $C_1$ and $C_2$, respectively. For $i\in\{1,2\}$, exactly two edges of $C_i$ in $G_{M''}$ are incident to the cut vertex $v_i$ of $C_i$; thus, in $M$, $v_i$ corresponds to an intersection point of at most two segments of $S_i$. Since $S_i$ contains at least three segments, there is a segment $s_i\in S_i$ which does not contain $v_i$ as an intersection point in $M$. Then, since $C_i$ is an outer cycle block, $s_i$ does not belong to any other circuit segment set of $M$, i.e., $s_i$ satisfies property $A)$. Furthermore, the connected components of $M\backslash s_i$ which do not contain segments of $S_i$ also do not contain segments of $M'$. However, as shown above, the connected components of $M\backslash s_i$ which do not contain segments of $M'$ are segment trees. Thus, $s_i$ satisfies property $B)$.
\qed

\begin{theorem}
Let $M$ be a segment cactus graph. Then
\begin{equation}
p(M) \leq 2(m(M)-k_1(M))-3k_2(M),
\label{eq99}
\end{equation}
\begin{equation}
c(M) \leq (m(M)-k_1(M))-2k_2(M),
\label{eq100}
\end{equation}
and these bounds are tight.
\label{thm1}
\end{theorem}
\proof
If $M$ is a segment forest, then $p\leq 2p-k_2\leq 2(m-k_1-k_2)-k_2=2(m-k_1)-3k_2$, where the first inequality follows from Observation \ref{obs1} (part 6.) and the second inequality follows from Observation \ref{obs2} (part 2.); this establishes the upper bound in (\ref{eq99}). Likewise, if $M$ is a segment forest, then the upper bound in (\ref{eq100}) follows from Observation \ref{obs1} (part 5.) and the fact that $c=0$. 
%Thus, we will henceforth assume that $M$ is not a segment forest, i.e., that $c \geq 1$, and hence $m\geq 3$. 
Thus, it remains to be shown that the upper bounds in (\ref{eq99}) and (\ref{eq100}) hold for the case when the segment cactus is not a segment forest, i.e., when $c \geq 1$, and hence $m\geq 3$. We will proceed by induction on $m$. Both inequalities clearly hold for $m=3$. Assume the inequalities hold for some $m \geq 3$ and let $M$ be a segment cactus with $m+1$ segments.

By Proposition~\ref{main_prop}, $M$ contains a segment $s_1$ which belongs to a single circuit segment set $S_1$, such that the connected components of $M\backslash s_1$ which do not contain segments of $S_1$ are segment trees. If $M\backslash s_1$ does not have any connected components which do not contain segments of $S_1$, let $s_*=s_1$. Note that in this case, deleting $s_*$ from $M$ decreases the number of intersection points by at most two, and the number of circuits by one. If $M\backslash s_1$ has at least one connected component $T$ which does not contain segments of $S_1$, then $T$ is a segment tree which can only intersect $s$ in a single point, since otherwise $s$ would be part of at least two circuits. If $T$ consists of a single segment, let $s_*$ be that segment. If $T$ contains at least two segments, then by Corollary \ref{cor1}, $T$ contains two segments $s_a$ and $s_b$, each having a single intersection point, such that removing either one of them from $T$ does not disconnect $T$. If neither $s_a$ nor $s_b$ intersect $s$, let $s_*=s_1$. If exactly one of $s_a$ and $s_b$ intersects $s$, let $s_*$ be the segment among $s_a$ and $s_b$ which does not intersect $s$. If both $s_a$ and $s_b$ intersect $s$, then $s$, $s_a$, and $s_b$ must all intersect in the same point; in this case, let $s_*=s_1$. In each of these cases, deleting $s_*$ from $M$ decreases the number of intersection points by at most one and does not affect the number of circuits.

Thus, the segment cactus $M\backslash s_*$ has $m$ segments, $p-i$ intersection points for some $i\in\{0,1,2\}$, and $c-t$ circuits for some $t\in\{0,1\}$. By the induction hypothesis, $p-i \leq 2(m-k_1)-3k_2$. Then, for the segment cactus $M$ with $m+1$ segments and $p$ intersections, we obtain $p \leq 2(m-k_1)-3k_2+i\leq 2(m-k_1)-3k_2+2=2(m+1-k_1)-3k_2$. Similarly, by the induction hypothesis, $c-i \leq (m-k_1)-2k_2$. Then, for the segment cactus $M$ with $m+1$ segments and $c$ circuits we obtain $c \leq (m-k_1)-2k_2+i\leq (m-k_1)-2k_2+1=(m+1-k_1)-2k_2$. This concludes the inductive step and establishes the inequalities. The inequalities in (\ref{eq99}) and (\ref{eq100}) hold with equality for all $m\geq 1$ for the construction shown in Figure \ref{22a25}.
\qed

\begin{figure}[h]
\begin{center}
\includegraphics[scale=0.45]{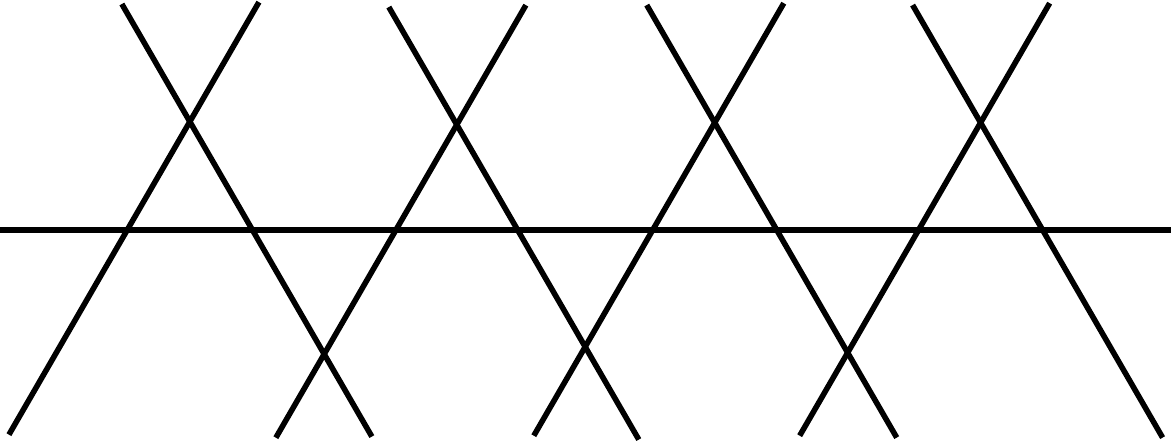}
\end{center}
\caption{A class of segment cactus graphs for which the bounds in (\ref{eq99}) and (\ref{eq100}) hold with equality.}
\label{22a25}
\end{figure}

\noindent Since a segment forest is a segment cactus graph, $p\geq k_2$ and $c\geq 0$ are tight lower bounds for segment cactus graphs. 

\subsection{Segment $K_3$-free graphs}

A \emph{$K_3$-free} graph\footnote{We use this nomenclature instead of \emph{triangle-free graph} in order to avoid confusion between geometric and graph theoretic triangles.} is a graph which has no subgraph isomorphic to $K_3$.

\begin{theorem}
Let $M$ be a segment $K_3$-free graph.  Then

\begin{align*}
&p(M)\leq {m(M)\choose 2}-(m(M)-2),\\
&c(M)\leq {m(M)-2\choose 2},
\end{align*}
and these bounds are tight.
\end{theorem}
\proof
Given a segment set $M$, let $t(M)$ denote the number of $K_3$ subgraphs of $G_M$, let $A(M)=p(M)-t(M)$, and let $B(M)=c(M)-t(M)$. When there is no scope for confusion, dependence on $M$ will be omitted. We will refer to the circuits of $M$ which correspond to $K_3$-subgraphs of $G_M$ as \emph{triangle circuits}. Let $M$ be an arbitrary segment $K_3$-free graph with $m$ segments. Let $M_0$ be obtained from $M$ by extending each segment $s$ of $M$ which has an endpoint that is also an intersection point by a small distance in the direction of that endpoint so that the endpoint is no longer an intersection point, but $s$ does not intersect any new segment. Note that $M_0$ is also a segment $K_3$-free graph (since $G_{M_0}$ is obtained by adding some leaves to $G_{M}$, which cannot create a $K_3$ subgraph), and that $p(M)=p(M_0)$ and $c(M)=c(M_0)$. Let the segments of $M_0$ be $s_1,\ldots,s_m$. We will transform $M_i$, $i\geq 0$, into $M_{i+1}$ by perturbing $s_i$ as follows. 

First, translate $s_i$ by a small distance so that none of its intersection points are shared with more than one other segment, and so that $s_i$ does not intersect any segments that it did not previously intersect. Since none of the endpoints in $M_i$ are intersection points, this can be done by choosing a small enough translation distance (which is also small enough that no endpoints become intersection points after the translation). For each intersection point of $s_i$ that was shared with more than one other segment before the translation, the number of new intersection points that are created as a result of the translation is one more than the number of new triangle circuits created (see Figure \ref{figure_k3_free}, top for an illustration). Each existing triangle circuit which intersects $s_i$ in a side or a point either remains a triangle circuit or is turned into non-triangle circuit through the translation. However, non-triangle circuits cannot disappear or be turned into triangle circuits through the translation, because doing so would require segments which did not previously intersect to intersect after the translation. Thus, $B$ does not decrease after the translation; moreover, since $p$ increases at least as much as $t$ after the translation, $A$ also does not decrease.

\begin{figure}[h]
\begin{center}
\includegraphics[scale=0.6]{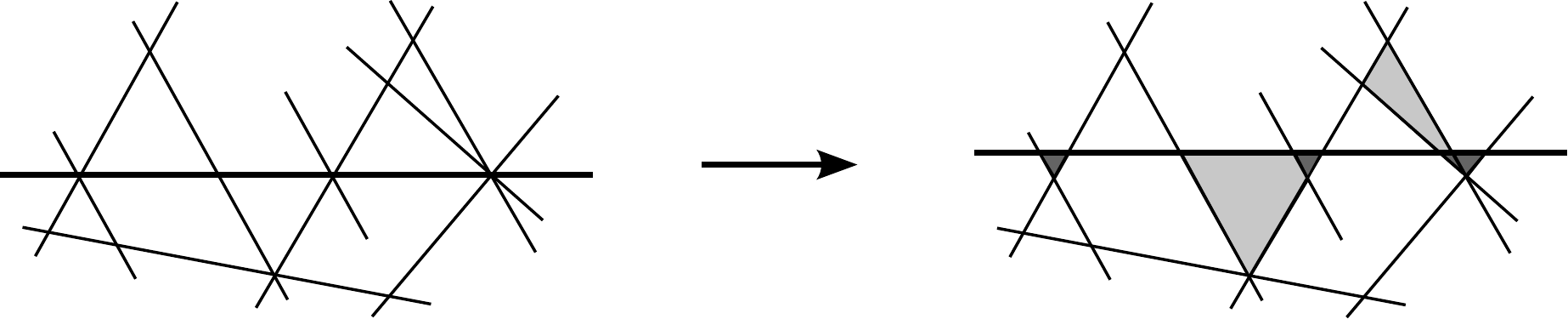}
\end{center}
\caption{The bold segment is translated by a small distance so that none of its intersection points share more than one other segment. The number of triangle circuits created by the translation (shaded dark) is no more than the number of intersection points created; some triangle circuits become non-triangle circuits (shaded light), but no non-triangle circuits disappear.}
\label{figure_k3_free}
\end{figure}

Next, rotate $s_i$ by a small (possibly zero) degree so that it is not parallel to any other segment in $M_i$. The degree can be chosen small enough so that each intersection point of $s_i$ remains an intersection point between the same two segments it was previously an intersection point between (note that the endpoints are not intersection points, so no intersection points will disappear if the degree of rotation is small enough). Thus, the topology of neither the circuits nor the intersection points is affected by this rotation, so $A$ and $B$ do not change. 

Next, extend $s_i$ from one endpoint until it intersects another segment; if the extended endpoint of $s_i$ intersects the new segment at an already-existing intersection point, translate $s_i$ by a small distance so that all of the previous intersection points and circuits (and their topologies) are preserved, but the extended endpoint of $s_i$ intersects the new segment in a point that was not previously an intersection point. As discussed previously, no non-triangle circuits disappear as a result of such a translation, if one is necessary. If the extension does not split any circuit into two circuits, then $B$ has not changed, $t$ has not changed, and $p$ has increased by one. If this extension does split a circuit into two circuits, then at most one of these two circuits is a triangle, since splitting a circuit into two triangle circuits requires at least one of the two intersection points of $s_i$ with the circuit to be a point where 3 segments meet (this is avoided by the translations). Thus, in either case, $p$ increases at least as much as $t$, so $A$ does not decrease. Moreover, regardless of whether the circuit that is split by $s_i$ was a triangle or a non-triangle, $B$ cannot decrease, since in either case at least one new non-triangle circuit is created and at most one non-triangle circuit disappears. Repeat this extension with both endpoints of $s_i$ until no new segments can be crossed, and then extend both endpoints by a sufficiently large distance so that all future extensions of segments will be able to intersect $s_i$. Call the resulting set of segments $M_{i+1}$. Since at each step of the perturbation, $A$ and $B$ either increase or remain unchanged, it follows that $A(M_i)\leq A(M_{i+1})$ and $B(M_i)\leq B(M_{i+1})$. 

By repeating the same perturbation with all segments, we obtain a sequence of segment sets $M_0,\ldots,M_m$ such that $A(M_0)\leq \ldots\leq A(M_m)$ and $B(M_0)\leq \ldots\leq B(M_m)$. In $M_m$, no segments are parallel, and no intersection point is shared by more than two segments; moreover, since the segments are long enough, each segment intersects every other segment. If each segment of $M_m$ is extended to a line, we obtain a set of lines $\mathcal{L}$ in general position; moreover, $A(\mathcal{L})=A(M_m)$ and $B(\mathcal{L})=B(M_m)$ since no new circuits or intersection points can be created by extending the segments. It is well-known that for a set of lines $\mathcal{L}$ in general position, $p(\mathcal{L})=\binom{m}{2}$ and $t(\mathcal{L})\geq m-2$. Hence, since $t(M_0)=0$, we have 
\[p(M)=p(M_0)-t(M_0)=A(M_0)\leq A(M_m)=A(\mathcal{L})=p(\mathcal{L})-t(\mathcal{L})\leq \binom{m}{2}-(m-2).\]
It is also well-known that the number of regions formed by a set of $m$ lines in general position is $\frac{m^2+m+2}{2}$ (this is known as the sequence of central polygonal numbers); this count includes $2m$ unbounded regions, so $c(\mathcal{L})=\frac{m^2+m+2}{2}-2m$. Then,
\[c(M)=c(M_0)-t(M_0)=B(M_0)\leq B(M_m)\leq B(\mathcal{L})=c(\mathcal{L})-t(\mathcal{L})\leq\frac{m^2+m+2}{2}-2m-(m-2).\]
Since $M$ was arbitrary, it follows that for any segment $K_3$-free graph, $p\leq \binom{m}{2}-(m-2)$ and $c\leq \frac{m^2+m+2}{2}-2m-(m-2)=\binom{m-2}{2}$. These bounds are tight for the following construction. Start with vertical parallel segments $r_0$ and $\ell_0$ which are long enough for future segments to intersect. Let $r_1$ be a segment crossing both $r_0$ and $\ell_0$ at an angle to the right. Then, for $i\geq 1$, we iteratively add segments $\ell_i$ and $r_{i+1}$ as follows. After a segment $r_i$ is added, find the intersection of $r_i$ and $r_{i-1}$, move down along $r_{i-1}$ a short distance, and place one endpoint of the segment $\ell_i$ there. From this endpoint, $\ell_i$ continues downward and to the left, at such an angle that it intersects every line already present, except for $r_i$. After a segment $\ell_i$ is added, find the intersection of $\ell_i$ and $\ell_{i-1}$, move down along $\ell_{i-1}$ a short distance, and place one endpoint of the segment $r_{i+1}$ there. From this endpoint, $r_{i+1}$ continues downward and to the right, at such an angle that it intersects every line already present, except for $\ell_i$. Newly added segments are long enough for all future segments to intersect; see Figure \ref{k3freefigure} for an illustration. There are two intersection points and zero circuits among $r_0$, $\ell_0$, and $r_1$; then, beginning with the fourth segment, each added segment intersects all-but-one of the existing segments. Thus, the total number of intersection points in this construction is $2+\sum_{i=4}^{m}(i-2)={m\choose 2}-(m-2)$, and the total number of circuits is $\sum_{i=4}^{m}(i-3)={m-2\choose 2}$.
\qed
\vspace{9pt}

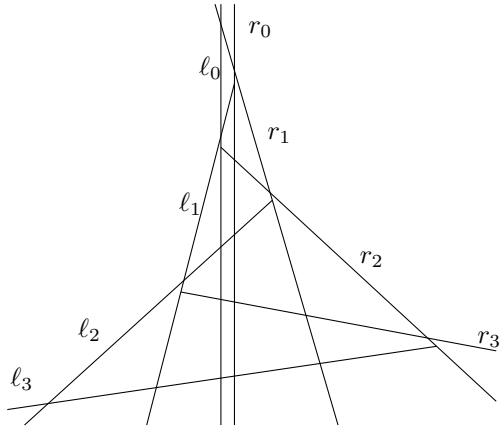
\begin{figure}[h]
\centering
\begin{tikzpicture}[scale=1]
\draw (2.8405, -5.60136) -- (2.8405, -.00137);
\draw (3.0205, -5.60136) -- (3.0205, -.00137);
\draw (2.76179, -.0014) -- (4.39746, -5.60132);
\draw (3.0205, -1.05637) -- (1.85385, -5.60136);
\draw (2.8405, -1.90637) -- (6.50047, -5.28574);
\draw (3.52389, -2.611) -- (.22912, -5.60136);
\draw (2.30819, -3.83136) -- (6.50049, -4.61421);
\draw (5.70762, -4.55367) -- (.0005, -5.39554);
\node[right] at (3.081924, -.342534) {$r_0$};
\node[right] at (3.336924, -1.752534) {$r_1$};
\node[right] at (4.561924, -3.394204) {$r_2$};
\node[right] at (6.128594, -4.462534) {$r_3$};
\node at (2.701924, -.850869) {$\ell_0$};
\node at (2.437173, -2.660869) {$\ell_1$};
\node[left] at (1.370259, -4.340874) {$\ell_2$};
\node[above] at (.196929, -5.254204) {$\ell_3$};
\end{tikzpicture}
\caption{Construction demonstrating the tightness of the upper bound on $p(M)$ for segment $K_3$-free graphs; in this example, $m=8$ and $p=22$.}
\label{k3freefigure}\end{figure}

Since a segment forest is a segment $K_3$-free graph, $p\geq k_2$ and $c\geq 0$ are tight lower bounds for general segment $K_3$-free graphs. A better lower bound can be derived for a segment $K_3$-free graph $M$ which has been trimmed (i.e., the corresponding graph $G_M$ has no leaves). In this case, $p= n\geq\frac{e+4}{2}\geq \frac{m+4}{2}$, where the first inequality follows from the fact that for any planar $K_3$-free graph $G$, $|E(G)|\leq 2(|V(G)|-2)$. This bound is tight, e.g., for a noncollinear straight-line embedding of the complete bipartite graph with parts of size $2$ and $n-2$.

\subsection{Segment maximal planar graphs}

A graph is \emph{maximal planar} if it is planar and adding any edge causes it to no longer be planar. Every face of a maximal planar graph (in any planar embedding) is a triangle. 

%\begin{proposition}
%Let $M$ be a segment maximal outerplanar graph.  Then
%\[p(M)\geq\left\lceil\frac{m(M)+3}{2}\right\rceil,\]
%\[c(M)\geq\left\lceil\frac{m(M)-1}{2}\right\rceil,\]
%and these bounds are tight.
%\end{proposition}
%\proof
%Maximal outerplanar graphs do not have any leaves, so for segment maximal outerplanar graphs, $J\subset P$, and hence $p=n$. For any maximal outerplanar graph $G$, $|E(G)|=2|V(G)|-3$. Thus, $p=\frac{e+3}{2}\geq\frac{m+3}{2}$ and since $p$ is an integer, this bound can be tightened to $p\geq \lceil \frac{m+3}{2}\rceil$. By Euler's formula and by the previous inequality, 
%\[c=1-n+e=1-n+(2n-3)=n-2=p-2\geq \frac{m+3}{2}-2=\frac{m-1}{2},\] and since $c$ is an integer, this bound can be tightened to $c\geq \lceil \frac{m-1}{2}\rceil$. The bounds on $p$ and $c$ hold with equality for any segment set obtained from a straight-line plane embedding of a maximal outerplanar graph in which the edges are drawn as non-collinear segments.
%\qed

\begin{proposition}
\label{prop_max_planar}
Let $M$ be a segment maximal planar graph.  Then
\[p(M)\geq\left\lceil\frac{m(M)+6}{3}\right\rceil,\]
\[c(M)\geq\left\lceil\frac{2m(M)-3}{3}\right\rceil,\]
and these bounds are tight. Moreover, there exist segment maximal planar graphs with 
$p=\theta(m^2)$ and $c=\theta(m^2)$.
\end{proposition}
\proof
Maximal planar graphs do not have any leaves, so for segment maximal planar graphs, $J\subset P$, and hence $p=n$.  For any maximal planar graph $G$, $|E(G)|=3|V(G)|-6$. Thus, $p=\frac{e+6}{3}\geq\frac{m+6}{3}$ and since $p$ is an integer, this bound can be tightened to $p\geq\lceil\frac{m+6}{3}\rceil$. By Euler's formula and by the previous inequality, 
\[c=1-n+e=1-n+(3n-6)=2n-5=2p-5\geq 2\left(\frac{m+6}{3}\right)-5=\frac{2m-3}{3},\] and since $c$ is an integer, this bound can be tightened to $c\geq \lceil \frac{2m-3}{3}\rceil$. The bounds on $p$ and $c$ hold with equality for any segment set obtained from a straight-line plane embedding of a maximal planar graph in which the edges are drawn as non-collinear segments. The set of segments obtained by triangulating an equilateral triangle with $x\geq 0$ segments parallel to each side, and connecting an external point to each intersection point on each side of the boundary of the triangle as in Figure \ref{fig_max_planar}, has $m=6x+9$ segments, and hence $p=\frac{(x+3)(x+2)}{2}+3=\frac{m^2}{72}+\frac{m}{6}+\frac{27}{8}=\theta (m^2)$ and $c=(x+1)^2+3(x+2)=\frac{m^2}{36}+\frac{m}{3}+\frac{7}{4}=\theta (m^2)$.
\qed

\begin{figure}[h]
\begin{center}
\includegraphics[scale=0.25]{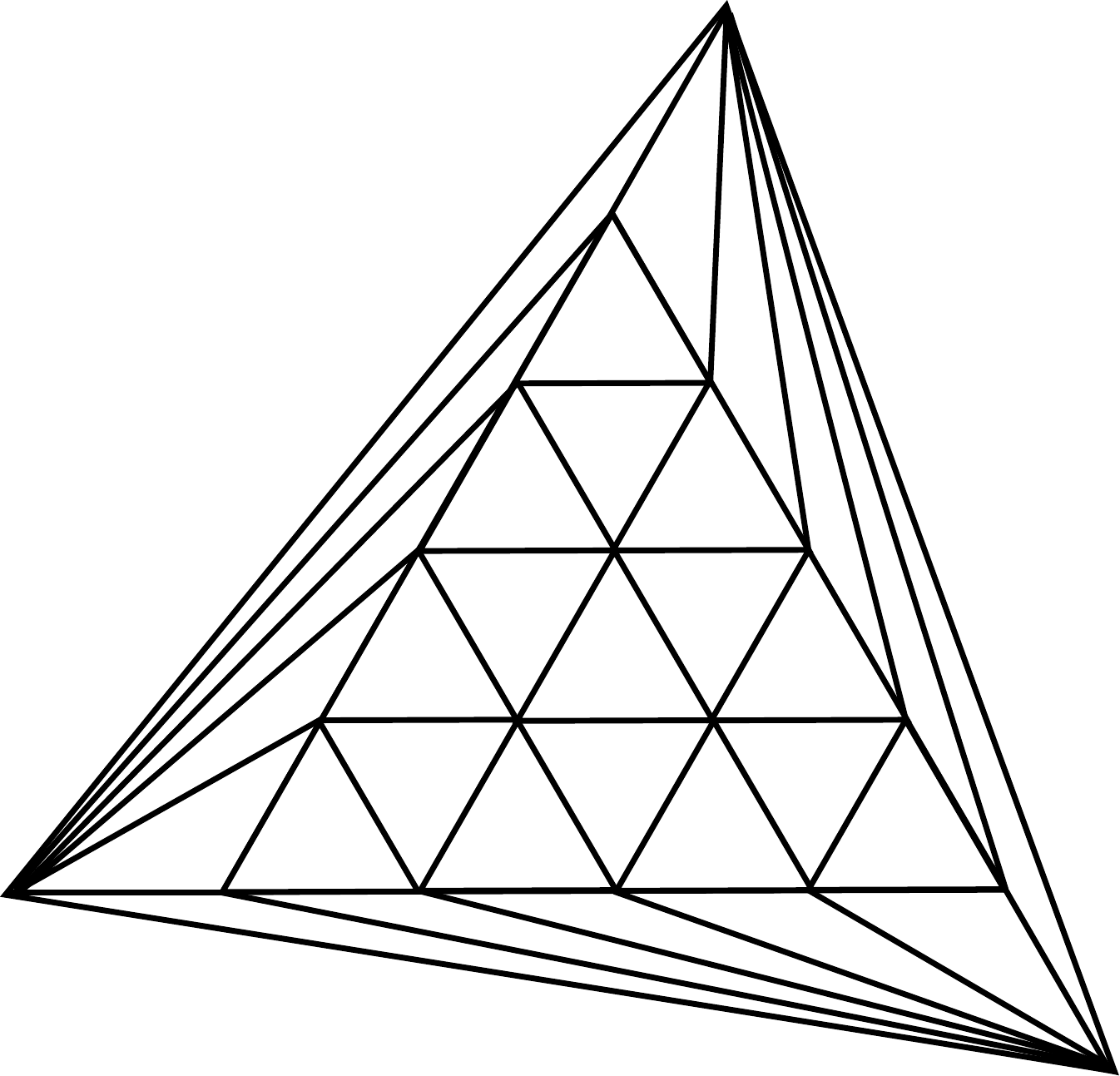}
\end{center}
\caption{Construction demonstrating the tightness of the asymptotic upper bound on $p(M)$ and $c(M)$ for segment maximal planar graphs.}
\label{fig_max_planar}
\end{figure}

\subsection{Buffon segments}
\label{section_buffon}

A set of segments $M=M(m,\ell)$ is called a \emph{Buffon set}\footnote{The nomenclature is derived from Buffon's Needle Problem, which investigates the probability that a needle will fall on a line when dropped on a sheet with equally spaced lines.} if it is produced by $m$ segments of length $\ell$ randomly placed in the unit square; dependence on $m$ and $\ell$ will be omitted when it is clear from the context. We assume the centers and angles of the segments are chosen uniformly at random. In this section, we investigate the expected number of intersection points in a Buffon set of segments as a function of $m$ and $\ell$, along with other structural properties.

\begin{proposition}
\label{prop_T_intersecting}
Let $M$ be a Buffon set of segments. The expected number of distinct subsets of $M$ of $t$ segments which mutually intersect is $O(m^t\ell^{2t-2})$.
\end{proposition}
\begin{proof}
Let $s_1,\ldots,s_t$ be segments in $M$; let $X$ be the event that $s_1,\ldots,s_t$ mutually intersect, and let $Y$ be the event that the centers of $s_{2}, \dots, s_{t}$ are within distance $\ell$ of the center of $s_{1}$. If the center of some $s_i$, $2\leq i\leq t$, is not within distance $\ell$ of the center of $s_1$, then $s_1,\ldots,s_t$ cannot mutually intersect. Thus, $X$ implies $Y$, and $Pr[X]\leq Pr[Y]$. Moreover, $Pr[Y] \leq (\pi \ell^2)^{t-1}$, since this is the probability that the centers of $s_2,\ldots,s_t$ lie in a disk with radius $\ell$ centered at the center of $s_1$. Thus, $Pr[X]=O(\ell^{2t-2})$. The number of distinct sets of $t$ segments which mutually intersect is equal to $\sum_{T\subset M,\, |T|=t} X_T$, where $X_T$ is the indicator random variable for the event that the segments in $T$ mutually intersect. By linearity of expectation and the fact that $E[X_T]=O(\ell^{2t-2})$, the expected number of distinct sets of $t$ segments which mutually intersect is $\sum_{T\subset M,\, |T|=t}E[X_T] = O(m^t\ell^{2t-2})$.

\end{proof}

\begin{theorem}
\label{thm_exp_p_buff}
Let $M$ be a Buffon set of segments. Then, $E[p(M)]=\theta(m^2 \ell^2)$. 
\end{theorem}

\begin{proof}
Let $X$ be the event that two segments $s_{1}, s_{2}$ of length $\ell$ intersect, let $Y$ be the event that the distance between the centers of $s_1$ and $s_2$ is at most $\ell / 2$, and let $Z$ be the event that the distance between the centers of $s_1$ and $s_2$ is exactly $\ell / 2$. Let $p_1=Pr[X]$, $p_2=Pr[X \cap Y]$, $p_3=Pr[X|Y]$, $p_4=Pr[Y]$, and $p_5=Pr[X|Z]$. Clearly $p_1\geq p_2$, and by definition of conditional probability, $p_2 = p_3 p_4$. Moreover, $p_3 \geq p_5$ since the probability that two segments intersect decreases with the distance between their centers. Also, note that $p_5$ is a constant independent of $\ell$, since the probability that two segments of length $\ell$ intersect given that their centers are $\ell/2$ apart is equal to the probability that two segments of length $x\ell$ intersect, given that their centers are $x\ell/2 $ apart, for any $x > 0$ (including $x=1/\ell$). Finally, in order for event $Y$ to occur, the center of $s_2$ would have to lie in a disk of radius $\ell/2$ centered at the center of $s_1$; since at least a quarter of such a disk intersects the unit square (this happens when the center of $s_1$ is in a corner of the unit square), it follows that $p_4 \geq \frac{1}{4} \pi (\ell / 2)^2$. Thus, $p_1 \geq p_2\geq p_4p_5\geq \frac{p_5 \pi}{16} \ell^2 = \Omega(\ell^2)$.

Now, $p(M)= \sum_{\{s_1,s_2\}\subset M} X_{\{s_1,s_2\}}$, where $X_{\{s_1,s_2\}}$ is the indicator random variable for the event that segments $s_1$ and $s_2$ intersect. From the above argument, $E[X_{\{s_1,s_2\}}]=\Omega(\ell^2)$. By linearity of expectation, $E[p(M)]=\sum_{\{s_1,s_2\}\subset M}E[X_{\{s_1,s_2\}}] = \Omega(m^2 \ell^2)$. Moreover, by Proposition \ref{prop_T_intersecting}, $E[p(M)]=O(m^2\ell^2)$; thus, $E[p(M)]=\theta(m^2\ell^2)$. 
\end{proof}

\begin{corollary}
Let $M$ be a Buffon set of segments of length at most $\left(\frac{a}{m}\right)^{3/4}$, for any constant $a>0$. Then, the expected number of $K_3$ subgraphs of $G_M$ is $O(1)$.
\end{corollary}

\begin{proof}
Each $K_3$ subgraph in $G_M$ corresponds to a distinct triple of mutually intersecting segments in $M$ (but not vice versa). By Proposition \ref{prop_T_intersecting}, the expected number of triples of mutually intersecting segments in $M$ is $O(m^3\ell^{4})$. Since $\ell\leq\left(\frac{a}{m}\right)^{3/4}$ and $a$ is a constant independent of $m$, it follows that the expected number of $K_3$ subgraphs in $G_M$ is $O(1)$.
\end{proof}

The \emph{complexity} of a set of segments $M$ is the sum of the vertices, edges, and bounded faces of $G_M$. Below we derive a bound on the expected complexity of a Buffon set of segments that is tight up to a constant factor.

\begin{theorem}
Let $M$ be a Buffon set of segments. The expected complexity of $M$ is $\theta(m^2 \ell^2+m)$. 
\end{theorem}

\begin{proof}
In a Buffon set, the expected number of points which are both endpoints of segments and intersection points is zero, and the expected number of intersection points where more than $2$ segments meet is zero. Thus, in expectation, $n(M) = p(M)+j(M)=p(M)+2m(M)$, and $e(M)=\sum_{s\in M} (\text{number of intersection points in } s \text{ plus } 1)= 2p(M)+m(M)$.

Let $M_1,\ldots,M_t$ be the connected components of $M$. By the argument above, in expectation, for $1\leq i\leq t$, $n(M_{i})=p(M_i)+2m(M_i)$ and $e(M_i)=2p(M_i)+m(M_i)$. By Euler's formula, the expected number of faces of $M_{i}$ (including the outer face) is $(2p(M_{i})+m(M_{i}))+2-(p(M_{i})+2m(M_{i})) = p(M_{i})-m(M_{i})+2$. Counting only the bounded faces and summing over $i$, we have $c(M)=\sum_{i=1}^{t} (p(M_i)-m(M_i)+1)=p(M)-m(M)+t$. Thus, by linearity of expectation, the expected complexity of $M$ is $E[n(M)]+E[e(M)]+E[c(M)]=4E[p(M)]+2m(M)+t$. Since by Theorem~\ref{thm_exp_p_buff}, $E[p(M)]=\theta(m^2 \ell^2)$, and since $1\leq t\leq m$, it follows that the expected complexity of $M$ is $\theta(m^2 \ell^2+m)$. 
\end{proof}

%\begin{corollary}
%For every constant $\alpha > 0$, there exists a constant $\beta > 0$ such that the expected number of colors required to color the intersection points of a random segment graph with $m$ segments of length $\beta/\sqrt{m}$ so that no segment has the same color repeated is at most $\alpha m$.
%\end{corollary}

Finally, we give necessary and sufficient conditions for a Buffon set of segments to have no intersections (i.e., for $G_M$ to be a perfect matching) with probability approaching $1$ as the number of segments increases. We begin with a technical lemma.

\begin{lemma}
\label{lemma_buff_seg}
Let $m$ and $t$ be positive integers, $m\geq 2$, $m\geq t$, and let $x \in [0, 1/m]$. Then, \[\prod_{i=1}^{t}(1-ix)\geq 1-\frac{t(t+1)}{2}x.\]
\end{lemma}

\begin{proof}
Let $f(x) = \prod_{i=1}^{t}(1-ix)$. Then, 
\begin{eqnarray*}
f'(x) &=& \sum_{1\leq i\leq t} (-i)\prod_{j\in \{1,\ldots,t\}\backslash\{i\}}(1-jx),\\ 
f''(x) &=& \sum_{1\leq i<j\leq t}  2ij \prod_{k\in \{1,\ldots,t\}\backslash\{i,j\}}(1-kx).
\end{eqnarray*}
%\begin{eqnarray*}
%f'(x) &=& \sum_{1\leq i\leq t} \frac{-i}{1-ix}f(x)=\sum_{1\leq i\leq t} (-i)\prod_{j\in \{1,\ldots,t\}\backslash\{i\}}(1-jx),\\ 
%f''(x) &=& \sum_{1\leq i\leq t}(-i)\frac{\sum_{1\leq j\leq t}(-j)\frac{f(x)(1-ix)}{(1-jx)}-f(x)(-i)}{(1-ix)^2}\\
%&=&\sum_{1\leq i\leq t}\sum_{1\leq j\leq t}\frac{ij}{(1-ix)(1-jx)}f(x)-\sum_{1\leq i\leq t}\frac{i^2}{(1-ix)^2}f(x)\\
%&=&\sum_{1\leq i<j\leq t}  2ij \prod_{k\in \{1,\ldots,t\}\backslash\{i,j\}}(1-kx).
%\end{eqnarray*}
By Taylor's Theorem, $f(x) = f(0)+ f'(0) x+\frac{f''(r)}{2} x^2$, for some $r \in (0, x)$. Note that $f(0) = 1$ and $f'(0) = -\frac{t(t+1)}{2}$. Moreover, since $r< x\leq 1/m$ and $k\leq t\leq m$, $f''(r)$ is a sum of products of nonnegative real numbers, so $f''(r)\geq 0$. Thus, $f(x)\geq 1-\frac{t(t+1)}{2}x$.
\end{proof}

\begin{theorem}
\label{lemma_no_intersections}
Let $M$ be a Buffon set of segments of length $\ell$. Then, as $m\rightarrow \infty$, $Pr[p(M)=0]\rightarrow 1$ if and only if $\ell=o(1/m)$.
\end{theorem}

\begin{proof}
Let the segments in $M$ be $s_{1}, \dots, s_{m}$, and for $1\leq t\leq m$, let $M_t=s_1\cup\ldots\cup s_t$. 
If $\ell = o(\frac{1}{m})$, then $Pr[p(M)>0]=Pr[p(M)\geq 1]\leq \frac{E[p(M)]}{1}\rightarrow 0$ as $m\rightarrow \infty$, where the last inequality follows from Theorem \ref{thm_exp_p_buff} and from Markov's inequality. Thus, if $\ell = o(\frac{1}{m})$, then $Pr[p(M)=0]\rightarrow 1$ as $m \rightarrow \infty$. 

Now, suppose that $\ell=\frac{\alpha}{m}$ for some $\alpha\in (0, \sqrt{1/\pi}]$. Let $X_t$ be the event that the centers of $s_{1}, \dots, s_{t-1}$ are all at least distance $\ell$ from each other. Since $X_t$ implies that $p(M_{t-1})=0$ with probability 1, and by the definition of conditional probability, we have 
\begin{equation}
\label{eqn_buff_cond}
Pr[(p(M_t)>0) \cap (p(M_{t-1})=0)] \geq Pr[(p(M_t)>0) \cap X_t]= Pr[X_t] \,Pr[(p(M_t)>0) \,|\, X_t].
\end{equation} 
The probability that the center of $s_{2}$ is at least $\ell$ away from the center of $s_{1}$ is at least the area of the unit square minus the area of a disk of radius $\ell$; hence,
\[Pr[X_3]\geq 1- \pi \ell^2.\]
Similarly, the probability that the center of $s_t$ is outside the disks of radius $\ell$ around the centers of $s_1,\ldots,s_{t-1}$ is at least $1-(t-1) \pi \ell^2$, so
\[Pr[X_{t+1}\,|\,X_{t}] \geq 1-(t-1) \pi \ell^2.\]
Thus, by the chain rule for probabilities, 
\begin{equation}
\label{eq_buffseg1}
Pr[X_t] =\prod_{i=3}^{t}Pr[X_i\,|\,X_{i-1}]\geq\prod_{i=1}^{t-2} (1- i \pi \ell^2) \geq 1-\frac{(t-1)(t-2)}{2}\pi \ell^2,
\end{equation}
where the last inequality follows from Lemma \ref{lemma_buff_seg}. Note that since $\alpha\leq \sqrt{1/\pi}$ and $\ell=\alpha/m$, it follows that $\pi\ell^2 \leq 1/m$ for all $m\geq 1$, so the conditions of Lemma \ref{lemma_buff_seg} are satisfied. 

Moreover, as shown in the proof of Theorem \ref{thm_exp_p_buff}, the probability that two segments of length $\ell$ intersect is at least $\frac{p_5\pi}{16}\ell^2$, where $p_5$ is the probability that two segments of length 1 intersect, given that their centers are distance $0.5$ apart. By the same argument, if the center of $s_{t}$ is within distance $\ell/2$ of the center of some other segment $s_i$, $1\leq i\leq t-1$, the conditional probability that $s_{t}$ intersects $s_i$ is at least $p_5$. Thus, 
\begin{equation}
\label{eq_buffseg2}
Pr[(p(M_t)>0) \,|\, X_t] \geq (t-1)\frac{p_5\pi}{16}\ell^2.
\end{equation}
Substituting \eqref{eq_buffseg1} and \eqref{eq_buffseg2} into \eqref{eqn_buff_cond}, we obtain
\[Pr[(p(M_t)>0) \cap (p(M_{t-1})=0)] \geq \left(1-\frac{(t-1)(t-2)}{2}\pi \ell^2\right)(t-1)\frac{p_5\pi}{16}\ell^2.\]
The event $(p(M)>0)$ is the disjoint union of the events $(p(M_t)>0) \cap (p(M_{t-1})=0)$,  $2\leq t\leq m$. Thus,  $Pr[p(M)>0]$ can be computed as follows:
\begin{eqnarray*}
Pr[p(M)>0]&=&\sum_{t = 2}^{m} Pr[(p(M_t)>0) \cap (p(M_{t-1})=0)]\\
&\geq&\frac{p_5\pi}{16}\ell^2 \sum_{t = 1}^{m-1} t\left(1-\frac{t(t-1)}{2} \pi \ell^2\right) \\
&=& \frac{p_5\pi}{16}\ell^2 \left( \sum_{t=1}^{m-1}t-\frac{\pi\ell^2}{2}\sum_{t=1}^{m-1}t^3+\frac{\pi\ell^2}{2}\sum_{t=1}^{m-1}t^2\right)\\
&=& \frac{p_5\pi}{16}\ell^2 \left( \frac{m(m-1)}{2}-\frac{\pi\ell^2}{2}\frac{(m-1)^2m^2}{4}+\frac{\pi\ell^2}{2}\frac{(m-1)m(2m-1)}{6}\right)\\
&\geq& \frac{p_5\pi}{16}\ell^2 \left(\frac{m(m-1)}{2} - \frac{\pi \ell^2m^4}{8}\right)\\
&=&\frac{p_5\pi}{16}\left(\frac{\alpha^2}{2}-\frac{\pi\alpha^4}{8}-\frac{\alpha^2}{2m}\right)>0,
\end{eqnarray*}
for $m$ sufficiently large. Thus, if $\ell = \frac{\alpha}{m}$ for $\alpha\leq\sqrt{1/\pi}$, then $Pr[p(M)>0]\not\rightarrow 0$ as $m\rightarrow \infty$. However, since increasing $\ell$ cannot decrease $Pr[p(M)>0]$, it follows that if $\ell = \Omega(\frac{1}{m})$, then $Pr[p(M)>0]\not\rightarrow 0$ as $m\rightarrow \infty$.
\end{proof}

\section{Erd\H{o}s-Faber-Lov\'asz conjecture for segments}
\label{section_efl}
The following is a long-standing conjecture of Erd\H{o}s, Faber, and Lov\'asz:
\begin{conjecture}[EFL Conjecture \cite{efl}]
Let $G$ be a graph consisting of $m$ copies of $K_m$, every pair of which has at most one vertex in common. Then, $\chi(G)=m$.
\end{conjecture}
\noindent 
Let $M$ be a set of curves (or lines, or segments). An \emph{EFL coloring of $M$ with $k$ colors} is a function $f:P(M)\rightarrow \{1,\ldots,k\}$ such that if two intersection points $p_1$, $p_2$ belong to the same curve (respectively line or segment), then $f(p_1)\neq f(p_2)$. The EFL Conjecture can be restated in a geometric form as follows:
\begin{conjecture}[EFL Conjecture]
Let $M$ be a set of $m$ curves such that every pair of curves has at most one point in common. Then, $M$ has an EFL coloring with $m$ colors.
\end{conjecture}
\noindent This conjecture can be relaxed to lines and segments instead of curves as follows:
\begin{conjecture}[Line EFL Conjecture]
\label{efl_line}
Let $M$ be a set of $m$ lines drawn in the plane. Then, $M$ has an EFL coloring with $m$ colors.
\end{conjecture}

\begin{conjecture}[Segment EFL Conjecture]
\label{efl_segment}
Let $M$ be a set of $m$ segments drawn in the plane. Then, $M$ has an EFL coloring with $m$ colors.
\end{conjecture}

In this section, we investigate for the first time the line version and the segment version of the EFL Conjecture; these are natural special cases of the EFL Conjecture, but have not previously been studied. We first show that Conjectures \ref{efl_line} and \ref{efl_segment} are equivalent in general. We then prove the conjectures are true for some special families of lines and segments, and investigate a related optimization problem. Note that the conjectures are not necessarily equivalent for families of lines versus families of segments corresponding to classes of planar graphs, since it is not the case that for every planar graph $G$ there exists a set of lines $M$ such that $G_M\simeq G$. Other geometric problems related to the EFL Conjecture have recently been investigated, sometimes framed in the context of hypergraphs; see \cite{hegde,huemer,klein,paul,romero} and the bibliographies therein.

\begin{proposition}
Conjecture \ref{efl_line} is true if and only if Conjecture \ref{efl_segment} is true.
\end{proposition}
\proof
Suppose Conjecture \ref{efl_line} is true, and let $M$ be a set of line segments drawn in the plane. Let $\widehat{M}$ be a set of lines obtained by replacing each segment $s\in M$ with a line passing through the endpoints of $s$, and merging collinear lines. Then $m(\widehat{M})\leq m(M)$, and $P(M)\subset P(\widehat{M})$. Since Conjecture \ref{efl_line} is true, $\widehat{M}$ has an EFL coloring $\widehat{f}:P(\widehat{M})\rightarrow \{1,\ldots,m(\widehat{M})\}$. Let $f:P(M)\rightarrow \{1,\ldots,m(M)\}$ be a coloring defined by $f(a)=\widehat{f}(a)$ for each $a\in P(M)$. Note that $\widehat{f}(a)\neq \widehat{f}(b)$ whenever $a$ and $b$ are on the same line in $\widehat{M}$, and that if two intersection points are on the same segment in $M$, they are on the same line in $\widehat{M}$; thus, it follows that $f(a)\neq f(b)$ for any $a,b\in P(M)$ lying on the same segment in $M$. Hence, $f$ is an EFL coloring of $M$ with $m(M)$ colors, so Conjecture \ref{efl_segment} is true.

Now suppose Conjecture \ref{efl_segment} is true, and let $M$ be a set of lines drawn in the plane. Let 
$\widehat{M}=M\cap \conv(P(M))$, i.e., $\widehat{M}$ is the set of segments obtained by taking the portion of each line which falls within $\conv(P(M))$. Then, $m(\widehat{M})\leq m(M)$ and $P(\widehat{M})=P(M)$. Since Conjecture \ref{efl_segment} is true, $\widehat{M}$ has an EFL coloring $\widehat{f}:P(\widehat{M})\rightarrow \{1,\ldots,m(\widehat{M})\}$. Let $f:P(M)\rightarrow \{1,\ldots,m(M)\}$ be a coloring defined by $f(a)=\widehat{f}(a)$ for each $a\in P(M)$. Since $\widehat{f}(a)\neq \widehat{f}(b)$ whenever $a$ and $b$ are on the same segment in $\widehat{M}$, and since if two intersection points are on the same line in $M$, they are on the same segment in $\widehat{M}$, it follows that $f(a)\neq f(b)$ for any $a,b\in P(M)$ lying on the same line in $M$. Thus, $f$ is an EFL coloring of $M$ with $m(M)$ colors, so Conjecture \ref{efl_line} is true.
\qed
\vspace{9pt}

In the clique version of the EFL Conjecture, clearly no fewer than $m$ colors can be used to color the graph, since $m$ colors are needed for each clique. However, in the geometric versions of the EFL Conjecture, it is possible to color the intersection points with far fewer than $m$ colors. Thus, it is natural to define the following optimization problem related to the EFL Conjecture.\\

\noindent \textsc{EFL-Coloring}\\
\textsc{Instance}: Set of curves $M$ such that any two intersect at most once; integer $k$\\
\textsc{Question}: Does $M$ have an EFL coloring with $k$ colors?\\

\noindent  We show below that \textsc{EFL-Coloring} is NP-Complete, even for a very restricted set of curves. We will assume that the equations specifying the curves can be evaluated in polynomial time, or that it is known which curves meet at each intersection point.

\begin{theorem}
\textsc{EFL-Coloring} is NP-Complete, even if $M$ is a set of segments where no five segments intersect in the same point.
\end{theorem}
\proof

Given a set of curves $M$, any two of which intersect at most once, and a function $f:P(M)\rightarrow \{1,\ldots,k\}$, it can be verified in polynomial time that $f$ is an EFL coloring of $M$. Thus, \textsc{EFL-Coloring} is in NP. 
%We will now show that \textsc{EFL-Coloring} is NP-hard, even when $M$ is a set of segments where no five segments intersect in the same point, by showing that this problem contains as a special case the following problem, which was shown to be NP-hard in \cite{}.
Let $G$ be a 4-regular planar graph equipped with a straight-line plane embedding where no two edges are drawn as collinear segments. Let $M$ be the set of segments comprising the edges of $G$. Since no two edges of $G$ are collinear in the embedding of $G$, the segments in $M$ intersect only at their endpoints. Moreover, since $G$ is 4-regular, every endpoint is an intersection point, and no five segments intersect in the same point. 
Thus, $P(M)=V(G)$ and $M=E(G)$, so a function $f:P(M)\rightarrow \{1,\ldots,k\}$ such that $f(a)\neq f(b)$ for each $a,b$ which both belong to some $s\in M$ is both an EFL coloring of $M$ and a proper coloring of $G$. Since finding a proper coloring of $G$ with $k$ colors is equivalent to finding an EFL coloring of $M$ with $k$ colors, it follows that \textsc{EFL-Coloring} contains \textsc{4-regular planar graph coloring} as a subproblem. Since the latter is known to be NP-Complete \cite{dailey}, it follows that \textsc{EFL-Coloring} is also NP-complete, even if $M$ is a set of segments where no five segments intersect in the same point.
\qed
\vspace{9pt}

\begin{corollary}
Let $M$ be a set of segments such that the embedding of $G_M$ induced by $M$ has no collinear edges. Then $M$ has an EFL coloring with 4 colors.  
\end{corollary}
\proof
By the Four Color Theorem, $G_M$ has a proper coloring $f:V(G_M)\rightarrow \{1,2,3,4\}$. Since the embedding of $G_M$ induced by $M$ has no collinear edges, each segment of $M$ corresponds to exactly one edge of $G_M$, and it follows that $f$ is also an EFL coloring of $M$.\qed

\begin{observation}
With probability 1, a Buffon set $M$ has an EFL coloring with $w(M)$ colors. 
\end{observation}
\proof
In a Buffon set, with probability 1, the number of intersection points where more than $2$ segments meet is zero. Then, if the segments in $M$ are labeled $1,\ldots,m$ and an intersection point between segments $i$ and $j$ is assigned color $(i+j) \mod m$, the resulting assignment is an EFL coloring (see \cite{hindman}). 
\qed
\vspace{9pt}

Below we show that every segment cactus graph has an EFL coloring with $m$ colors. In fact, we show something slightly stronger. Let $\mathcal{A}$ be any set of 3 segments with 3 intersection points. Clearly $\mathcal{A}$ has an EFL coloring with 3 colors. We show that $\mathcal{A}$ is the only segment cactus which requires $m$ colors for an EFL coloring; for all other segment cactus graphs, $m-1$ colors are sufficient, and sometimes necessary.

\begin{theorem}
Let $M$ be a segment cactus different from $\mathcal{A}$. Then $M$ has an EFL coloring with $m-1$ colors; moreover, there exist segment cactus graphs different from $\mathcal{A}$ which do not have an EFL coloring with $m-2$ colors.
\end{theorem}
\proof
If $M$ is disconnected, colors can be used independently in each of its connected components. Thus, assume without loss of generality that $M$ is connected. We will prove the claim by induction on $c(M)$. If $c(M)=0$, then $M$ is a segment tree, and $p\leq m-1$. Thus, each intersection point can be colored with a distinct color to obtain an EFL coloring with $m-1$ colors. 

\begin{claim}
\label{claim1}
Let $M$ be a set of segments which can be partitioned into a segment set $M'$ and a segment tree $T$ such that $M'$ and $T$ intersect in exactly one point. If $M'$ has an EFL coloring with $k$ colors, then $M$ has an EFL coloring with $k+m(T)$ colors.
\end{claim}

\proof
Let $f:P(M')\rightarrow \{1,\ldots,k\}$ be an EFL coloring of $M'$. Let $a_1,\ldots,a_t$ be all intersection points in $M$ which belong to segments of $T$, where without loss of generality, $a_1=M'\cap T$ is the point where $M'$ and $T$ intersect. Note that $t\leq m(T)$, and that $a_1$ may or may not be in $P(M')$; in either case, $P(M)=(P(M')\backslash \{a_1\})\dot\cup \{a_1,\ldots,a_t\}$. Then, $f':P(M)\rightarrow \{1,\ldots,k+m(T)\}$ given by 
\begin{equation*}
f'(a)=\begin{cases}
f(a) \qquad&\text{ if }  a\in P(M')\backslash \{a_1\}\\
k+i\qquad&\text{ if }  a=a_i, 1\leq i\leq t
\end{cases}
\end{equation*}
is an EFL coloring of $M$ with $k+m(T)$ colors.
\qed
\vspace{9pt}

If $c(M)=1$ and $M\neq \mathcal{A}$, then $M$ has at least 4 segments. Let $S$ be the circuit segment set of the circuit of $M$. If $S$ contains exactly 3 segments, let $s\notin S$ be a segment of $M$ which intersects $S$. It is easy to check that, regardless of where $s$ intersects $S$, the segment set $M':=S\cup \{s\}$ has an EFL coloring with 3 colors. If $S$ contains more than 3 segments, let $M':=S$; then $P(M')$ has an EFL coloring with at most 3 colors (by alternating two colors and using a third color if $m(M')$ is odd). In both cases, since $M$ has a single circuit, the connected components of $M\backslash M'$ are segment trees $T_1,\ldots,T_k$, each of which intersects $M'$ in exactly one point. By Claim \ref{claim1}, since $M'$ has an $EFL$ coloring with $m(M')-1$ colors, $M$ has an EFL coloring with $m(M)-1$ colors. 

%Let $\mathcal{X}$ be any set of segments obtained by taking a set of segments in $A$ which has at least two points and adding a segment which passes through such that when $(\mathcal{X})$ is trimmed, $p(\mathcal{X})=4$ any set of 4 segments with 5 intersection points and 2 cycles. 

Now suppose all segment cactus graphs with $c(M)=t\geq 1$ which are different from $\mathcal{A}$  have an EFL coloring with $m-1$ colors. Let $M$ be a connected segment cactus with $c(M)=t+1$. By Proposition \ref{main_prop}, $M$ has a segment $s$ which belongs to a single circuit segment set $S$ and the connected components of $M\backslash s$ which do not contain segments of $S$ are segment trees $T_1,\ldots,T_k$. Note that for each $i\in\{1,\ldots,k\}$, exactly one segment of $T_i$ intersects $s$, since otherwise $s$ would be part of more than one circuit. Let $M'=M\backslash (T_1\cup\ldots\cup T_k)$. By construction, $s$ intersects exactly two segments $x$ and $y$ of $M'$, respectively in the points $a_x$ and $a_y$. Moreover, $M'\backslash s$ is a segment cactus with $c(M')=t$. 

If $M'\backslash s$ is of type $\mathcal{A}$, then $M'$ has 4 segments, one of which is $s$, and $s$ intersects two of the segments of $M'\backslash s$ outside the circuit of $M'\backslash s$ (since this is the only way to add a segment to $\mathcal{A}$ to produce two circuits and remain a segment cactus). Then, it is easy to see that $M'$ has an EFL coloring with $m(M')-1=3$ colors. 

%Let $a_1,a_2,a_3$ be the intersection points of $M'\backslash s$. Without loss of generality, suppose $a_x$ is on the segment containing $a_2$ and $a_3$, and $a_y$ is on the segment containing $a_1$ and $a_3$. See Figure ?? for an illustration. Then, the function $f:\{1,2,3,4\}\rightarrow\{a_1,a_2,a_3,a_4,a_5\}$ given by $f(a_1)=1$, $f(a_2)=2$, $f(a_3)=1$, $f(a_x)=1$, $f(a_y)=2$ is an EFL coloring of $M'$ with $m(M')-1$ colors. 

Now, suppose that $M'\backslash s$ is different from $\mathcal{A}$; then, by the inductive hypothesis, $M'\backslash s$ has an EFL coloring $f:P(M'\backslash s)\rightarrow \{1,\ldots,m(M'\backslash s)-1\}$. Clearly, in $M'\backslash s$, $x$ and $y$ each have at most $m(M'\backslash s)-1$ intersection points. Suppose for contradiction that both $x$ and $y$ have exactly $m(M'\backslash s)-1$ intersection points. Then, every segment in $M'\backslash \{s,x\}$ intersects $x$, and every segment in $M'\backslash \{s,y\}$ intersects $y$. In particular, $x$ and $y$ intersect each other at some point $a$. Let $b_x$ and $c_x$ be the endpoints of $x$, where $b_x$ is closer to $a$ than to $a_x$; let $b_y$ and $c_y$ be the endpoints of $y$, where $b_y$ is closer to $a$ than to $a_y$.
Then, no segment can intersect $x$ and $y$ so that it has one endpoint in $\overline{c_xa}$ and the other in $\overline{c_ya}$, since that would create two circuits in $M'$ sharing more than a single point. Similarly, no segment can intersect $x$ and $y$ so that it has one endpoint in $\overline{b_ya}$ and the other in $\overline{c_xa}$, or one endpoint in $\overline{b_xa}$ and the other in $\overline{c_ya}$. If exactly one segment intersects $x$ and $y$ with one endpoint in $\overline{b_ya}$ and the other in $\overline{b_xa}$, then $M'\backslash s$ is of type $\mathcal{A}$, a contradiction. Otherwise, if more than one segment intersects $x$ and $y$ with one endpoint in $\overline{b_ya}$ and the other in $\overline{b_xa}$, this would create two circuits in $M'$ sharing more than a single point, a contradiction. Finally, any segments of $M'\backslash s$ intersecting $x$ and $y$ at $a$ do not add to the number of intersection points of $x$ and $y$; thus, in all cases it follows that both $x$ and $y$ cannot have exactly $m(M'\backslash s)-1$ intersection points. 

Without loss of generality, suppose $x$ has at most $m(M'\backslash s)-2$ intersection points. If $a_x$ is an intersection point in $M'\backslash s$, let $q=f(a_x)$. Otherwise, if $a_x$ is not an intersection point in $M'\backslash s$, let $q$ be a color of the EFL coloring $f$ that does not appear on the segment $x$, i.e., $q\in \{1,\ldots,m(M'\backslash s)-1\}$ such that for all $a\in P\cap x$, $f(a)\neq q$. Then, $f':P(M')\rightarrow \{1,\ldots,m(M')-1\}$ defined by 
\begin{equation*}
f'(a)=\begin{cases}
f(a) \qquad&\text{ if }  a\in P(M'\backslash s)\backslash\{a_y\}\\
q\qquad&\text{ if }  a=a_x\\
m(M')-1\qquad&\text{ if }  a=a_y
\end{cases}
\end{equation*}
is an EFL coloring of $M'$. By Claim \ref{claim1}, since $M'$ has an $EFL$ coloring with $m(M')-1$ colors, $M$ has an EFL coloring with $m(M)-1$ colors. 

For the family of segment cactus graphs pictured in Figure \ref{22a25}, the horizontal segment intersects $m-1$ segments, so this family of segment cactus graphs does not have an EFL coloring with $m-2$ colors. 
\qed
\vspace{9pt}

\noindent Given a set of segments (or curves, or lines) $M$ and $s\in M$, let $P(s,M)$ denote the set of intersection points contained in $s$, let $p(s,M)=|P(s,M)|$, and let $w(M)=\max_{s\in M}\{p(s,M)\}$. When there is no scope for confusion, dependence on $M$ will be omitted. Note that for any segment set $M$, $w\leq m-1$.

\begin{observation}
For any set $M$ of segments (or curves, or lines), at least $w(M)$ colors are necessary for an EFL coloring of $M$. 
\end{observation}

\noindent We will now show two examples for which $w(M)$ colors are also sufficient for an EFL coloring. 

\begin{proposition}
\label{prop_efl_tree}
Let $M$ be a segment tree. Then $M$ has an EFL coloring with $w(M)$ colors.
\end{proposition}

\proof
Let $s$ be an arbitrary segment of $M$; arbitrarily assign the colors $\{1,\ldots,p(s,M)\}$ to its intersection points. Let $s'$ be a segment which intersects $s$ at point $a$ and let $f(a)$ be the color of $a$; arbitrarily assign the colors $\{1,\ldots,p(s',M)\}\backslash f(a)$ to the intersection points of $s'$ other than $a$. Repeat this process by successively picking a segment which intersects a segment whose intersection points are already colored, until all intersection points in $M$ are colored. The resulting coloring is an EFL coloring, since using colors $\{1,\ldots,p(s',M)\}\backslash f(a)$ for each new segment $s'$ is enough to assure that each of the $p(s',M)-1$ uncolored intersection points of $s'$ receive a distinct color. Moreover, by construction, only colors $\{1,\ldots,w(M)\}$ are used over all segments. 
\qed

\begin{theorem}
\label{thm_efl_k3free}
Let $M$ be a line $K_3$-free graph. Then $M$ has an EFL coloring with $w(M)$ colors. 
\end{theorem}

\proof
If all lines of $M$ are parallel, then there are no intersection points and we are done. Suppose that the lines of $M$ can be separated into two sets $S_1$ and $S_2$ of mutually parallel lines. 
%(i.e., the lines in $M$ form a grid)
If $|S_1|=1$ or $|S_2|=1$, then $p(M)=w$, so $w$ colors clearly suffice for an EFL coloring; thus, suppose without loss of generality that $w=|S_1|\geq |S_2|\geq 2$.
%Without loss of generality, suppose the lines in $S_1$ each have $w$ intersection points, and the lines in $S_2$ each have $w'\leq w$ intersection points. 
Let $\ell\in S_1$ and $\ell'\in S_2$ be lines which form two sides of the convex hull of $P(M)$; order the other lines in $S_1$ and $S_2$ according to their distance from $\ell$ and $\ell'$, respectively. Let $f:P(M)\rightarrow \{1,\ldots,w\}$ be defined by $f(a_{ij})=(i+j) \mod w$, where $a_{ij}$ is  the intersection point which is in the $i^\text{th}$ line in $S_1$ and the $j^\text{th}$ line in $S_2$, according to the ordering specified above. Then, $f$ is an EFL coloring of $M$ with $w$ colors, since if $a_{ij_1}$ and $a_{ij_2}$ are two points on the same line, $f(a_{ij_1})=(i+j_1) \mod w\neq (i+j_2) \mod w=f(a_{ij_2})$.

Now, suppose the lines of $M$ cannot be separated into two sets of mutually parallel lines. If all lines in $M$ intersect at the same point, then clearly $p(M)=1=w(M)$, and we are done. Otherwise, there must be three lines in $M$ which intersect each other in different points, forming a triangle. Let $\ell_1,\ldots,\ell_m$ be the lines in $M$. Without loss of generality, suppose $\ell_1,\ell_2,\ell_3$ are lines which intersect in three different points. For $1\leq i\leq m$, let $M_i=\ell_1\cup\ldots\cup \ell_i$. Suppose $M_i$ is not $K_3$-free for some $i\in\{3,\ldots,m-1\}$, and let $T_i$ be the triangle in $M_i$ formed by the lines corresponding to the edges of some $K_3$ subgraph of $G_{M_i}$. If $\ell_{i+1}$ does not pass through the interior of $T_i$, then $T_i$ sill induces a $K_3$ subgraph in $G_{M_{i+1}}$. If $\ell_{i+1}$ passes through the interior of $T_i$, then it creates at least one new triangle along with two of the lines forming $T_i$, which induces a new $K_3$ subgraph in $G_{M_{i+1}}$. Thus, in either case, if $M_i$ is not a line $K_3$-free graph, then $M_{i+1}$ is not a line $K_3$-free graph. Since $M_3$ is not $K_3$-free, it follows by induction that $M$ is not $K_3$-free. 
\qed

\section{Concluding remarks}
\label{section_conclusion}

In this paper, we derived tight bounds on the number of intersection points and circuits of different families of segment sets. Such bounds on $p$ and $c$ in terms of $m$ can yield better bounds on the time and space complexities of existing algorithms. In particular, in Section \ref{section_bounds}, we considered segment Halin graphs, segment cactus graphs, segment $K_3$-free graphs, and segment maximal planar graphs. These classes of segments are mostly non-overlapping and thus constitute a significant part of all sets of segments; for instance, segment Halin graphs and segment $K_3$-free graphs are disjoint, as are segment maximal planar graphs and segment cactus graphs (for $m\neq 3$). Some other interesting families to consider are segment bipartite planar graphs and segment maximal outerplanar graphs. By a similar reasoning as in Proposition \ref{prop_max_planar}, it can be shown that for a segment maximal outerplanar graph, $p\geq\left\lceil(m+3)/2\right\rceil$ and 
$c\geq\left\lceil(m-1)/2\right\rceil$. However, we have not been able to find exact or asymptotic tight upper bounds on $p$ and $c$. A construction of segment maximal outerplanar graphs with $p=2m-6$ and $c=2m-8$ is shown in Figure \ref{fig_outerplanar}, but in general, this construction is not the best possible. However, we conjecture that the upper bounds for both $p$ and $c$ are linear in $m$.

\begin{figure}[h!]
\begin{center}
\includegraphics[scale=0.35]{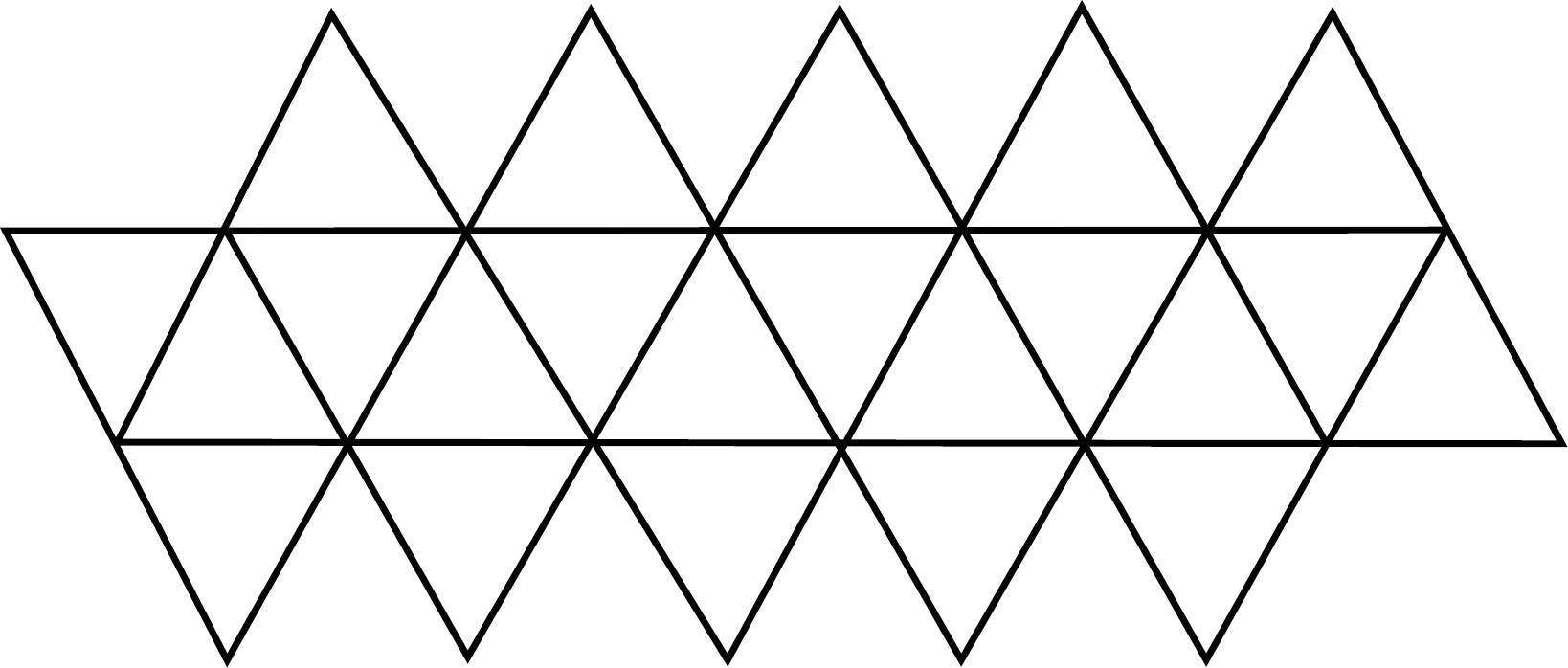}
\end{center}
\caption{A class of segment outerplanar graphs.}
\label{fig_outerplanar}
\end{figure} 

In Section \ref{section_buffon}, we investigated randomly generated sets of segments with fixed length. A related direction for future work is to explore properties of Buffon segment sets with non-uniform lengths; for example, the lengths of the segments could be random variables with a given probability distribution. The expected value of other parameters of $G_M$ (such as independence number, maximum matching, etc.) could also be explored, for a Buffon set $M$ with uniform or non-uniform length segments. 

In Section \ref{section_efl}, we introduced a geometric variant of the EFL Conjecture, and proved it true for several classes of segments and lines. These results are largely disjoint from previous partial results on the EFL Conjecture, and the geometric formulation of the EFL Conjecture allows us to approach it using geometric tools and techniques. 
A parameter $\chi_E(M)$ related to the optimization problem \textsc{EFL-Coloring} can be defined as the smallest value of $k$ such that $M$ has an EFL coloring with $k$ colors. Proposition \ref{prop_efl_tree} and Theorem \ref{thm_efl_k3free} showed that $\chi_E(M)=w(M)$ for segment trees and line $K_3$-free graphs. It would be interesting to determine whether $\chi_E(M)$ can be arbitrarily higher than $w(M)$, i.e., whether there exists a family of segment sets for which $w(M)=o(\chi_{E}(M))$. An example of a segment set $M$ for which $\chi_{E}(M)$ is strictly greater than $w(M)$ is the set of segments corresponding to the edges of a straight-line embedding of $K_4$. Deriving other bounds on $\chi_{E}(M)$ approaching $m(M)$ would be a step to proving the EFL Conjecture for general sets of segments and lines. Computational approaches for EFL coloring could also be of independent interest. For example, the following integer program could be used to compute $\chi_{E}(M)$ for an arbitrary set of curves $M$.
\begin{align*}
\min &\sum_{1\leq k\leq p} y_k\\
&\sum_{1\leq k\leq p} x_{ik}=1  &\forall i\in \{1,\ldots,p\}\\
&x_{ik}\leq y_k &\forall i,k\in \{1,\ldots,p\}\\
&x_{ik}+x_{jk}\leq 1 &\forall \{i,j\}\subset P(s,M), s\in M, k\in\{1,\ldots,p\}\\
&y_k,x_{ik}\in \{0,1\} &\forall i,k\in\{1,\ldots,p\}
\end{align*}
Here $y_k=1$ if color $k$ is used, and $x_{ik}=1$ if intersection point $i$ gets color $k$. Results like the ones derived in Section \ref{section_bounds} can be  used to bound the number of constraints in this integer program.

\section*{Acknowledgements}
 
We thank Edinah Gnang for suggesting the study of Buffon segments and Stephen Hartke for several useful discussions. This work is partially supported by NSF-DMS grants 1603823 and 1604458.

%--------------------------

\end{document}